\documentclass[11pt]{amsart}
\usepackage{dsfont,mathdots}
\usepackage{amsthm}
\usepackage{enumitem}
\usepackage{mathtools}
\usepackage{moreenum}
\usepackage{amssymb,verbatim}
\usepackage{upgreek}
\usepackage{xspace,cmap}
\usepackage{csquotes}
\usepackage{stmaryrd} 
\usepackage[colorlinks,citecolor=blue,urlcolor=black,linkcolor=black]{hyperref}
\usepackage{bbm}
\usepackage{todonotes}
\usepackage{array}
\usepackage{booktabs}

\usepackage{mathtools}

\usepackage{pst-node}
\usepackage{tikz}
\usetikzlibrary{positioning, arrows.meta}

\usepackage{tikz-cd}

\renewcommand{\restriction}{\mathbin\upharpoonright}    
\newtheorem*{rep@theorem}{\rep@title}
\newcommand{\newreptheorem}[2]{%
\newenvironment{rep#1}[1]{%
 \def\rep@title{#2 \ref{##1}}%
 \begin{rep@theorem}}%
 {\end{rep@theorem}}}
\newreptheorem{theorem}{Theorem}
\newtheorem*{theorem*}{Theorem}
\newtheorem*{maintheorem*}{Main Theorem}
\newtheorem*{corollary*}{Corollary}
\newtheorem*{definition*}{Definition}
\newtheorem{mainthm}{Main Theorem}

\newtheorem{theorem}{Theorem}[section]

\newtheorem{prop}[theorem]{Proposition}
\newtheorem{claim}[theorem]{Claim}

\newtheorem{lemma}[theorem]{Lemma}
\newtheorem{cor}[theorem]{Corollary}

\newtheorem{question}[theorem]{Question}
\newtheorem{fact}[theorem]{Fact}
\newreptheorem{corollary}{Corollary}

\theoremstyle{definition}
\newtheorem{example}[theorem]{Example}
\newtheorem{definition}[theorem]{Definition}
\newtheorem{notation}[theorem]{Notation}
\newtheorem{conv}[theorem]{Convention}

\theoremstyle{remark}
\newtheorem{remark}[theorem]{Remark}

\newcommand*\axiomfont[1]{\textsf{\textup{#1}}}
\newcommand\zfc{\axiomfont{ZFC}}

\newcommand\ben[1]{\marginpar{Ben: #1}}

\newcommand\seba[1]{\marginpar{Seba: #1}}

\renewcommand{\P}{\mathbb{P}}
\newcommand{\Q}{\mathbb{Q}}

\newcommand{\U}{\mathcal{U}}
\newcommand{\pkl}{P_\kappa(\lambda)}

\DeclareMathOperator{\range}{Im}

\DeclareMathOperator{\id}{id}
\DeclareMathOperator{\fin}{fin}
\DeclareMathOperator{\FIN}{FIN}
\DeclareMathOperator{\CH}{CH}

    \newcommand{\one}{\mathop{1\hskip-3pt {\rm l}}}

\makeatletter
\newcommand{\tpitchfork}{
  \vbox{
    \baselineskip\z@skip
    \lineskip-.52ex
    \lineskiplimit\maxdimen
    \m@th
    \ialign{##\crcr\hidewidth\smash{$-$}\hidewidth\crcr$\pitchfork$\crcr}
  }
}
\makeatother

\def\s{\subseteq}
\def\forces{\Vdash}

\DeclareMathOperator{\col}{Col}
\DeclareMathOperator{\Add}{Add}

\DeclareMathOperator{\rng}{Im}

\DeclareMathOperator{\ro}{\mathcal{RO}}

\DeclareMathOperator{\cf}{cf}

\DeclareMathOperator{\ord}{Ord}
\DeclareMathOperator{\add}{Add}

\newcommand{\dom}{\mathop{\mathrm{dom}}\nolimits}

\newcommand{\Col}{\mathop{\mathrm{Col}}}

\title[Strongly Compact Prikry]{On the intermediate models of strongly compact Prikry forcing}
\author[Benhamou]{Tom Benhamou}
\address[Benhamou]{Hill Center for Mathematical Sciences, Rutgers University, Piscataway, New Jersey USA}
\email[Benhamou]{tom.benhamou@rutges.edu}
\urladdr{https://sites.math.rutgers.edu/~tb822/}
\author[Thei]{Sebastiano Thei}
\address[Thei]{IMPAN, Warsaw, Poland}
\email[Thei]{thei91.seba@gmail.com}
\author[Weltsch]{Ben-Zion Weltsch}
\address[Weltsch]{Hill Center for Mathematical Sciences, Rutgers University, Piscataway, New Jersey USA}
\email[Weltsch]{ben.w@rutgers.edu}

\subjclass[2020]{}

\begin{document}
\begin{abstract}
    We analyze the intermediate models of the strongly compact Prikry forcing. 
    We exhibit a simple combinatorial property which, for a given supercompact cardinal $\kappa$, characterizes all projections of the strongly compact Prikry forcing using $\kappa$-complete fine measures. Considering level-by-level results, if $\kappa$ is $2^\lambda$-strongly compact, we characterize the forcings of size $\leq\lambda$ which are   projections of that $\lambda$-strongly compact Prikry forcing.
    Our characterization generalizes several known results, including those from \cite{Benhamou_Hayut_Gitik} and folklore results regarding the class of $\kappa$-distributive forcing notions which are embedded into the supercompact Prikry forcing. 
    Fixing a $\kappa$-complete fine measure $\mathcal{U}$ on $P_\kappa(\lambda)$, we also provide Rudin-Keisler like criteria for the existence projections from the strongly compact Prikry forcing with $\mathcal{U}$. 
    Finally, we prove that among all projections of the $\lambda$-strongly compact Prikry forcing, the class of forcings of cardinality $\lambda$ are exactly those for which there is a projection map which depends only on the stem of the Prikry condition. We also give partial results regarding projections of arbitrary cardinality.
\end{abstract}
\maketitle

\section{Introduction}

Understanding a mathematical structure often narrows down to understanding its substructures. 
In set theory, and more particularly in forcing theory, the structure of subforcings of a given forcing can be translated to the analysis of intermediate models of generic extensions by that forcing. 
These intermediate extensions provide insight into the generic evolution of the forcing, and how sets are being added to the ground model. 
This perspective has proven especially fruitful in identifying rigidity, minimality, and universality phenomena among forcing notions.

A number of classical results illustrate this theme. Maharam \cite{Maharam} showed that subforcings of Cohen forcing are again the Cohen forcing, revealing a strong form of rigidity. 
An extreme example of a forcing with a rigid structure of intermediate models was given by Sacks  \cite{Sacks1971} who constructed the first example of a minimal forcing, in the sense that it admits no proper intermediate models. 
The structure of intermediate models for products of forcing notions remains more subtle; for instance, it is a long-standing open problem whether the product of two random forcings admits nontrivial intermediate extensions i.e. other than a single random extension and a single Cohen extension.

On the other extreme, some forcing notions are universal for certain classes, that is, a generic extension by such a forcing absorbs generic extensions from every forcing in that class.
The most basic example is that of the L\'{e}vy collapse $\col(\kappa,\lambda)$ which absorbs every $\kappa$-strategically closed forcing of size $\leq\lambda$ (see for example \cite[14.1]{Cummings_Handbook}). 
Our focus in this paper is on another universal forcing -- the $\lambda$-supercompact and $\lambda$-strongly compact Prikry forcing. 
This forcing was introduced by Magidor in \cite{Magidor1977} to resolve an old problem about failure of the Singular Cardinal Hypothesis at $\aleph_{\omega}$, but has also been one of the most important tools in the analysis of singular cardinal combinatorics \cite{Magidor1977b,MagidorShelah1994,gitik-sharon,neeman-tpsch,dima_tpalephomega}. Our main result is a simple combinatorial level-by-level classification of the intermediate models of the $\lambda$-strongly compact Prikry forcing of cardinality $\lambda$. 
To explain our classification, we recall two folklore results:
\begin{itemize}
    \item Every $\kappa$-distributive forcing of cardinality $\lambda$ is absorbed by the $2^\lambda$-supercompact Prikry forcing (see 
    \cite{Gitik2010}).
    \item If $W\leq_{RK} U$, the Prikry forcing with $U$ projects onto Prikry forcing with $W$ (see \ref{thm: seed implies weak projection}). 
\end{itemize}
We introduce the class of $(\omega,\kappa)$-predistributive forcing notions (see Definition \ref{def: predistributive} and Definition \ref{def ditributive}) and show that:

\begin{tikzpicture}[
    implies-arrow/.style={double, double equal sign distance, -{Implies}, shorten >=1pt, shorten <=1pt, thick}
]

    \node (2dist)    {$(\kappa,2)$-distributive};
    \node (predist) [right=of 2dist] {$(\omega,\kappa)$-predistributive};
    \node (blank) [right=2cm of predist] {};

    \node (kdist)  [above right=0.2cm and -1cm of blank] {$\kappa$-distributive};
    \node (sigmapr) [below=0.2cm of blank] {$\Sigma$-Prikry};
    

    \draw[implies-arrow] (predist) -- (2dist);
    \draw[implies-arrow]  (kdist) -- (predist);
    \draw[implies-arrow] (sigmapr) -- (predist);

\end{tikzpicture}

Then our main result of this paper is the following, which appears as Corollary \ref{LevelByLevel Cor}.
\begin{mainthm}
     Suppose that $\kappa$ is $2^\lambda$-strongly compact. 
     The following are equivalent for any forcing $\mathbb{Q}$ of cardinality $\leq\lambda$:
    \begin{enumerate}
        \item $\mathbb{Q}$ is $(\omega,\kappa)$-predistributive.
        \item There is a $\kappa$-complete fine ultrafilter $\mathcal{U}$ over $P_\kappa(\lambda)$ such that $\mathbb{Q}$ is a projection of $\mathbb{P}_{\mathcal{U}}$. 
    \end{enumerate}
\end{mainthm}
Although the theorem is stated in the context of large cardinals, the direction $(2) \Rightarrow (1)$ can be proved in ZFC.
We then deduce the following characterization of intermediate models of the supercompact and strongly compact Prikry forcing:
\begin{repcorollary}
    {cor: no bounded subsets} If $\kappa$ is strongly compact (supercompact), then the following are equivalent for any forcing notion $\mathbb{Q}$:
    \begin{enumerate}
        \item $\mathbb{Q}$ is $(\omega,\kappa)$-predistributive.
        \item There is a fine $\kappa$-complete  (normal) measure $\mathcal{U}$  such that $\ro(\mathbb{Q})$ is a projection of the $\mathbb{P}_{\mathcal{U}}$.
    \end{enumerate}
\end{repcorollary}

Intermediate models of Prikry-type forcings have been of independent interest. 
For example, they provide a vital framework to iterate distributive forcings (see \cite{Gitik2010}). 
This line of research includes several recent developments, starting with the work of Gitik-Kanovei-Koepke \cite{PrikryCaseGitikKanKoe} which proved that the only intermediate models of the classical Prikry forcing with a normal measure $U$ are Prikry-extensions for the same $U$. 
Koepke-R\"{a}sch-Schlicht \cite{MinimalPrikry} showed that a tree Prikry forcing can be minimal in the sense of Sacks. More recently, the first author and Gitik extended these results to Magidor forcing with $o(\kappa)<\kappa^+$ \cite{TomMoti,partOne,Parttwo}, and together with Hayut \cite{Benhamou_Hayut_Gitik} studied the tree Prikry forcing. Regarding forcings of larger cardinality the first author and Gitik showed that $\add(\kappa,\kappa^+)$ is consistently a projection of the tree Prikry forcing \cite{BENHAMOU_GITIK_2024}.
Gitik \cite{GitikOnCompactCardinals} showed that the extender-based Prikry forcing can also absorb all $\kappa$-distributive forcings of cardinality $\kappa$ (see also Poveda-Hayut \cite{HayutForthcoming-HAYTGP}). 
Very recently, Ben-Neria, Goldberg, and Kaplan \cite{BenNeria_Goldberg_Kaplan} showed that every $(\kappa,2)$-distributive  forcing is of Prikry type. 

The basic tool to analyze intermediate models is to establish a connection between the existence of a forcing projection and the Rudin-Keisler order or  the Kat\v{e}tov order (Def \ref{def rk}). 
These results can be viewed as a generalizations of the results from \cite{Benhamou_Hayut_Gitik}:
\begin{repcorollary}{corollary: Katetov dist}
    Let $\lambda$ be regular and let $\mathbb{Q}$ be a forcing of cardinality $\leq\lambda$ and let $\mathcal{U}$ be a $\kappa$-complete fine ultrafilter over $P_\kappa(\lambda)$. Then the following are equivalent:
     \begin{enumerate}
         \item There is   $n<\omega$ such that $\mathcal{F}(\mathbb{Q})\leq_K \mathcal{U}^n$.
         \item $\mathbb{Q}$ is $\sigma$-distributive and a   weak projection of $\mathbb{P}_{\mathcal{U}}$.
     \end{enumerate}
\end{repcorollary}
For non-distributive forcing, we obtain the following Rudin-Keisler like characterization:
\begin{mainthm}
    Let $\mathbb{Q}$ be a forcing notion of cardinality $\leq\lambda$ and let $\mathcal{U}$ be a fine $\kappa$-complete ultrafilter on $P_\kappa(\lambda)$.  Then the following are equivalent:
    \begin{enumerate}
        \item There is a sequence $q_0\geq q_1\geq...\geq q_n\geq...$ such that for every $n$, $q_n\in j_{\omega}(\mathbb{Q})$ diagonalizes $j^{\mathcal{U}}_\omega"\mathcal{F}_{\mathds{\mathbb{Q}}}$.
        \footnote{By this we mean that for each $n$, $A \in j^\mathcal{U}_\omega `` \mathcal{F}_{\Q}$ iff $q_n \in A$.}
        \item $\ro(\mathbb{Q})$ is a projection of $\mathbb{P}_{\mathcal{U}}$.
    \end{enumerate}
\end{mainthm}
Using ideas from the theory of $\Sigma$-Prikry forcing, namely the \emph{$\kappa$-trace property}, we prove the combinatorial criterion for being a projection of supercompact Prikry forcing.
\begin{mainthm}
    Every  $(\omega,\kappa)$-predistributive forcing whose filter of dense open sets is generated by ${\leq}\lambda$-many sets is a projection of the $\lambda$-supercompact Prikry forcing $\P_\mathcal{U}$ for every normal, fine $\mathcal{U}$ over $P_\kappa(\lambda)$.
\end{mainthm}
In the last section, namely Theorem \ref{Tukey-type omega Theorem} we are able to characterize the projections of $\mathbb{P}_{\mathcal{U}}$ of cardinality ${\leq}\lambda$ in terms of generating sets of the generic filter and also by the projection depending only on the stem:
\begin{reptheorem}{Tukey-type omega Theorem}
    Let $\mathbb{Q}$ be any forcing. Then the following are equivalent:
    \begin{enumerate}
        \item There is a dense $D\subseteq\mathbb{P}_{\mathcal{U}}$ and a forcing projection $\pi:D\to\mathbb{Q}$ which only depends on the stem of the Prikry condition.
        \item $\Q$ has a dense subset of size at most $\lambda$ and there is a projection $\pi : \P_\mathcal{U} \to \ro(\Q)$.
    \end{enumerate}
\end{reptheorem}
The above theorem shows that in order to construct generics for forcings of cardinality greater than $\lambda$, we will have to face more complicated ``shapes" of generating sets. The natural language to describe those shapes is  the Tukey order \cite{Tukey} and cofinal types (see Definition \ref{Def: Tukey}). After we prove some useful equivalences between order properties of the generic filter and combinatorial properties of the forcing, and after presenting several examples, we will show that, quite naturally, the critical Tukey-type is that of the ultrafilter (as computed in the generic extension): 
\begin{repcorollary}{cor: shape}
    If $\mathbb{P}_{\mathcal{U}}$ weakly projects to $\mathbb{Q}$, then $\Vdash_{\mathbb{P}_{\mathcal{U}}} \mathcal{U}\geq_T \dot{G}_{\mathbb{Q}}$
\end{repcorollary}
We will be able to derive some corollaries about projections of the Prikry forcings. Finally, we show the following simple corollary which rules out a projection to some yardstick forcings of greater cardinality. 
\begin{repcorollary}{larger projections}
    Suppose that $\mathcal{U}$ is a fine $\kappa$-complete ultrafilter over $P_\kappa(\kappa^+)$. Then $\mathbb{P}_{\mathcal{U}}$ cannot weakly project to the following forcings: $\add(\kappa,\kappa^{++})$, $\add(\kappa^+,\kappa^{++})$ and $\add(\kappa^{++},1)$. Moreover, $\mathbb{P}_{\mathcal{U}}$ cannot add fresh subsets to $\kappa^{++}$.
\end{repcorollary}
This paper is organized as follows:
\begin{itemize}
    \item In section \ref{Section: Preliminaries} we review the basic theory of projections of forcings, ultrafilters and the strongly and supercompact Prikry forcing.
    \item In section \ref{section: RK} we prove our results regarding the Rudin-Keisler like characterization of the equivalence of projection.
    \item In section \ref{sec: predistributive Forcings} we introduce the class of $(\omega,\kappa)$-predistributive forcings and show it characterizes the projections of the strongly compact Prikry forcing.
    \item In section \ref{sec: Tukey type} we prove the characterization using projections which depend only on the stem (Theorem \ref{Tukey-type omega Theorem}), and include a discussion about projections of larger cardinality. 
    \item In section \ref{section questions} we list some open problems and further directions.
\end{itemize}
\section{Preliminaries}\label{Section: Preliminaries}
\subsection{Projections}

Following the set-theoretic tradition, our notation for (forcing) posets will be $\mathbb{P}, \mathbb{Q}$, etc. As customary members $p$ of a poset $\mathbb{P}$ will be referred to as \emph{conditions}. Given two conditions $p,q\in \mathbb{P}$ we will write $``q\leq p$'' as a shorthand for \emph{$``q$ is stronger than $p$''} (equivalently, \emph{$``q$ extends $p$''}). We shall denote by $\mathbb{P}/p$ the subposet of $\mathbb{P}$ whose universe is  $\{q\in \mathbb{P}: q\leq p\}$. Similarly, $X/p:=\{x\in X: x\leq p\}$, for any $X\s \mathbb{P}$. Two conditions $p$ and $q$ are said to be \emph{compatible} if there is $r\in \mathbb{P}$ such that $r\leq p,q$. Finally, a forcing poset $\mathbb{P}$ is \emph{separative} if whenever $p\in\mathbb{P}$ does not extend $q\in\mathbb{P}$ there is an extension of $p$ that is incompatible with $q$.\footnote{In symbols, $\forall p,q\in\mathbb{P}\big(p\nleq q\Longrightarrow\exists r\leq p(r\perp q)\big).$} For any forcing poset $\mathbb{P}$, there is a separative poset $\Bar{\mathbb{P}}$, called the \emph{separative quotient of} $\mathbb{P}$, such that $\mathbb{P}$ and $\Bar{\mathbb{P}}$ are \emph{forcing equivalent} as forcing posets, in the sense
that both yield the same generic extension. In light of this, whenever we talk about forcing with a forcing poset $\mathbb{P}$, we actually mean that we force with its separative quotient $\Bar{\mathbb{P}}$.

  Our notation for $\mathbb{P}$-names will be $\tau$, $\sigma$, $\dot{a}$, $\dot{b}$, etc. If $\tau$ is a $\mathbb{P}$-name for an object in the ground model and $p$ is a condition, then
we say that $p$ \emph{decides} $\tau$ if there is $x\in V$ such that $p\Vdash\check{x}=\tau$. We write $p\parallel_{\mathbb{P}}\tau$ (or simply $p\parallel\tau$) for ``$p$ decides $\tau$''. Similarly, for any statement $\varphi$ of the forcing language, we write $p\parallel\varphi$ for ``$p$ decides $\varphi$''; i.e., $p\Vdash\varphi$ or $p\Vdash\neg\varphi$.
let us recall some basic properties of forcing notions which will be considered in this paper.

We denote by $\add(\kappa,\lambda)$ the usual Cohen forcing for adding $\lambda$-many subsets to $\kappa$ with ${<}\kappa$-approximations and the L\'{e}vy collapse $\col(\kappa,\lambda)$ which adds a surjection $f:\kappa\to\lambda$ with ${<}\kappa$-approximations. 

Recall that a forcing $\mathbb{P}$ has the $\lambda$-cc if every antichain has size less than $\lambda$. It is $(\mu,\lambda)$-centered if there is a decomposition $\mathbb{P}=\bigcup_{\alpha<\lambda}\mathbb{P}_\alpha$ such that for every $\alpha$, $\mathbb{P}_\alpha$ is ${<}\mu$-directed i.e. any $A\in [\mathbb{P}_\alpha]^{<\mu}$ is bounded in $\mathbb{P}_\alpha$. $\mathbb{P}$ is called $\lambda$-closed if every $\gamma<\kappa$ and every decreasing sequence $\langle p_\alpha\mid \alpha<\gamma\rangle$ of conditions in $\mathbb{P}$ has a lower bound in $\mathbb{P}$. 
\begin{definition}\label{def ditributive}
We say that a forcing $\mathbb{P}$ is \textit{$(\kappa,\mu)$-distributive} if for every $\gamma<\kappa$ and every function $f:\gamma\to \mu$ in a generic extension by $\mathbb{P}$ must belong to the ground model. We say that $\mathbb{P}$ is \textit{$\kappa$-distributive} if it is $(\kappa,\infty)$-distributive. 
\end{definition}
The above definition is usually giving in its equivalent form for Boolean algebras (see \cite{JECH198411}).

In what follows, we state some folklore facts about projections. For a comprehensive account of this topic, we recommend the excellent expositions in \cite{abraham2009proper, shelah2017proper} and \cite{MonEsk}.
\begin{definition}\label{def: projections}
        Let $\mathbb{P}$ and $\mathbb{Q}$ be forcing posets. \begin{enumerate}
         \item\label{def wp} A \emph{weak projection} is a map $\pi\colon \mathbb{P}\rightarrow\mathbb{Q}$ such that: \begin{enumerate}
            \item For all $p,p'\in\mathbb{P}$, if $p\leq p'$ then $\pi(p)\leq \pi(p')$;
            \item\label{clause 2 wp} For all $p\in\mathbb{P}$ there is $q\leq \pi(p)$  such that for all $q'\in \mathbb{Q}$ with $q'\leq q$ there is $p'\leq p$ such that $\pi(p')\leq q'$.\footnote{This definition comes from \cite{foreman1991generalized}.}
        \end{enumerate} 
        \item A  weak projection $\pi\colon \mathbb{P}\rightarrow\mathbb{Q}$ is called a \emph{projection} if $\pi(\one_{\mathbb{P}})=\one_{\mathbb{Q}}$ and in Clause \eqref{clause 2 wp} above the condition $q$ can be taken to be $\pi(p)$;\footnote{Recall that the weakest condition of a forcing poset $\mathbb{P}$ is customarily denoted by $\one_{\mathbb{P}}$ -- or simply by $\one$ if there is no confusion.} i.e., for all $p\in\mathbb{P}$ and $q\in\mathbb{Q}$ with $q\leq\pi(p)$, there is $p'\leq p$ such that $\pi(p')\leq q$. 
        \end{enumerate} 
        If there is a (weak) projection from $\mathbb{P}$ into $\mathbb{Q}$, we say that $\mathbb{P}$ \emph{(weakly) projects into} $\mathbb{Q}$. 
    \end{definition}
    The notion of a weak projection is well established and appears frequently in the literature. For instance, Foreman and Woodin \cite{foreman1991generalized} employed it to address the fact that Supercompact Radin forcing does not project into standard Radin forcing. Similar issues arise with Supercompact Prikry forcing (see \cite{kafkoulis1994consistency} and \cite{dimonte2023solovay}), AIM forcing (see \cite{cummings2018ordinal}), and Merimovich forcing (see \cite{PovedaThei2024baire}). If $\pi:\mathbb{P}\rightarrow\mathbb{Q}$ is a weak projection and $G\s\mathbb{P}$ is $V$-generic then the pointwise image of $G$ via $\pi$ generates a $V$-generic filter $H\s\mathbb{Q}$. To make this precise, we briefly review the relevant notation. Given a function $f:A\rightarrow B$, we denote by $\dom(f)$ the domain of $f$, while $\range(f)$ denotes the range of $f$; i.e., $\range(f):=\{b\in B:\exists a\in A\ (f(a)=b)\}$. If $C\s A$ and $D\s B$, stipulate $f``C:=\{f(c): c\in C\}$ and $f^{-1}[D]:=\{a\in A: f(a)\in D\}$.
    \begin{fact}[{\cite[Proposition 2.8]{foreman1991generalized}}]\label{fact: basics of weak projections}
        Let $\mathbb{P}$ and $\mathbb{Q}$ be forcing notions and $G$ be $\mathbb{P}$-generic.  If $\pi\colon \mathbb{P}\rightarrow\mathbb{Q}$ is a weak projection, then the upwards closure of the set $\pi``G$
            (namely, $\{q\in \mathbb{Q}: \exists p\in G\;\pi(p)\leq q\}$)
            is $V$-generic .
    \end{fact}
    \begin{conv}
          We will identify $\pi``G$ with its upwards closure.
    \end{conv}
    The following fact provides a useful characterization of being a (weak) projection.
    \begin{fact}\label{fact: equivalences for projections}
    Let $\pi:\mathbb{P}\to\mathbb{Q}$ be any order-preserving function between two separative forcings.
    
    \begin{itemize}
        \item The following are equivalent:
        \begin{enumerate}
            \item $\pi$ is a weak projection.
            \item $\pi$ satisfies $(b')$ where:
            \begin{enumerate}
                \item [$(b')$] For all $p\in\mathbb{P}$ there is $p^*\leq p$  such that for all $q'\in \mathbb{Q}$ with $q'\leq \pi(p^*)$ there is $p'\leq p$ such that $\pi(p')\leq q'$.\footnote{This definition comes from \cite[$\S3.6$]{cummings2018ordinal}.}    \end{enumerate}
        \item For each dense open $D\s\mathbb{Q}$, the preimage $\pi^{-1}[D]$  is dense.

        \end{enumerate}
        \item  The following are equivalent:
    \begin{enumerate}
        \item $\pi$ is a projection.
        \item  $\pi$ is a weak projection and for every $p\in\mathbb{P}$, $\pi``(\mathbb{P}/p)$ is dense below $\pi(p)$.
    \end{enumerate}
    \end{itemize}
\end{fact}
Next, we provide a characterization of (weak) projections via generic filters. Accordingly, we briefly review complete Boolean algebras. For notational convenience, we are going to identify $\mathbb{P}$ with its isomorphic copy in $\ro(\mathbb{P})$, the \emph{regular open Boolean algebra} of $\mathbb{P}$ (minus its least element $\mathbf{0}$). Bearing in mind this (usual) identification, $\mathbb{P}$ is dense in $\ro(\mathbb{P})$ and $\mathbb{P}$-names are also $\ro(\mathbb{P})$-names. In particular, if $\mathbb{Q}$ is a forcing poset and $\rho\colon\ro(\mathbb{P})\rightarrow\ro(\mathbb{Q})$ is a projection, then so is its restriction to $\mathbb{P}$. Moreover, if $G$ is $\ro(\mathbb{P})$-generic then $G\cap\mathbb{P}$ is $\mathbb{P}$-generic. Conversely, if $G$ is $\mathbb{P}$-generic, then its upwards closure in $\ro(\mathbb{P})$, i.e. $\{b\in\ro(\mathbb{P}): \exists p\in G\ p\leq b\}$, is $\ro(\mathbb{P})$-generic. 

    While a weak projection $\pi\colon\mathbb{P}\to\mathbb{Q}$ induces a $\mathbb{Q}$-generic filter from a $\mathbb{P}$-generic filter (Fact \ref{fact: basics of weak projections}), a projection $\pi:\mathbb{P}\rightarrow\mathbb{Q}$ additionally allows to lift a $\mathbb{Q}$-generic filter $H$ into a $\mathbb{P}$-generic $G$ such that $\pi``G=H$. The way to obtain such a $G$ is via the \emph{quotient forcing}. Specifically, given a projection $\pi: \mathbb{P}\rightarrow\mathbb{Q}$ and a $\mathbb{Q}$-generic filter $H$ one defines the \emph{quotient forcing $\mathbb{P}/H$} as the subposet of $\mathbb{P}$ with universe $\{p\in \mathbb{P}: \pi(p)\in H\}$.\footnote{Note that $\mathbb{P}/H$ may not be separative but as we pointed out $\mathbb{P}/H$ is identified with its separative quotient.} Then every $\mathbb{P}/H$-generic $G$ (over $V[H]$) is $\mathbb{P}$-generic (over $V$) and $\pi``G=H$. This yields the following characterizations.

        \begin{fact}\label{fact: equivalence with generics and weak projections} Let $\mathbb{P}$ and $\mathbb{Q}$ be forcing posets. The following are equivalent:
    \begin{enumerate}
        \item $\P$ weakly projects into $\ro(\Q)$.
        \item There is a $\mathbb{P}$-name $\dot{H}$ such that $\one\Vdash_{\P}``\dot{H}$ is $\mathbb{Q}$-generic$"$.
        \item For every $\P$-generic $G$ there is a $\Q$-generic $H$ with $H\in V[G]$.
    \end{enumerate}
\end{fact}
\begin{fact}\label{fact: proj and complete emb}
    Let $\mathbb{P}$ and $\mathbb{Q}$ be forcing posets. Then the following are equivalent:
\begin{enumerate}
    \item $\P$ projects into $\ro(\Q)$.
    \item There is a projection from a dense subset of $\mathbb{P}$ into $\mathbb{Q}$.
    \item\label{item 3: proj and com emb} There is a $\mathbb{P}$-name $\dot{H}$ for a $\mathbb{Q}$-generic filter such that for all $q \in \mathbb{Q}$, there is $p \in \mathbb{P}$ such that $p \Vdash_{\mathbb{P}} \check{q} \in \dot{H}$.
    \end{enumerate}
\end{fact}
Appealing to Fact \ref{fact: proj and complete emb}, we can isolate sufficient conditions to ensure the existence of a projection (see Proposition \ref{Proposition: turnining a weak projection to a projection} below). To prove this, the following standard lemmas will be useful. 
\begin{lemma}[Maximal Principle]\label{lem: max princ}
    Let $\mathbb{P}$ be a forcing notion. For any formula $\varphi(x,y_1,\dots,y_n)$ and $\mathbb{P}$-names $\dot{u}_1, \dots, \dot{u}_n$, if $p \in \mathbb{P}$ is a condition such that $p \Vdash \exists x \, \varphi(x, \dot{u}_1, \dots, \dot{u}_n)$, then there exists a $\mathbb{P}$-name $\dot{\tau}$ such that
\[ p \Vdash \varphi(\dot{\tau}, \dot{u}_1, \dots, \dot{u}_n) \]
\end{lemma}
\begin{lemma}[Mixing Lemma]\label{lem: mixing}
    Let $\mathbb{P}$ be a forcing notion. Let $A = \{ p_i : i \in I \}$ be an antichain in $\mathbb{P}$, and let $\{ \dot{\tau}_i : i \in I \}$ be a family of $\mathbb{P}$-names. Then there exists a $\mathbb{P}$-name $\dot{\tau}$ such that $p_i \Vdash \dot{\tau} = \dot{\tau}_i$, for all $i \in I$.
\end{lemma}

\begin{prop}\label{Proposition: turnining a weak projection to a projection}
    Let $\mathbb{P}$ and $\mathbb{Q}$ be two forcing notions. Let $\mu$ be any cardinal with $\mu\geq|\mathbb{Q}|$ and suppose that $\mathbb{P}$ is not $\mu$-cc. If \begin{equation}\label{eq: one forces}
        \one\forces_{\mathbb{P}}``\forall q\in\mathbb{Q}\exists h\, (h \text{ is $\mathbb{Q}$-generic and } q\in h)",
    \end{equation} then $\P$ projects into $\ro(\Q)$. 
\end{prop}
\begin{proof}
        Since $\P$ is not $\mu$-cc, there is a maximal antichain $A=\{p_\alpha:\alpha<\nu\}\s\mathbb{P}$, for some cardinal $\nu\geq\mu$. Since $\nu\geq|\mathbb{Q}|$, we may pick an enumeration $\{q_\alpha: \alpha<\nu\}$ (possibly with repetitions) of $\mathbb{Q}$. Combining \eqref{eq: one forces} and the Maximal Principle (Lemma \ref{lem: max princ}), for each $\alpha<\nu$ there is a $\mathbb{P}$-name $\dot{h}_\alpha$ such that $p_\alpha\Vdash_{\mathbb{P}}``q_\alpha\in \dot{h}_\alpha$ and $\dot{h}_\alpha$ is $\mathbb{Q}$-generic''. By the Mixing Lemma (Lemma \ref{lem: mixing}), there is a name $\dot{h}$ such that $p_\alpha\Vdash_{\mathbb{P}}\dot{h}=\dot{h}_\alpha$, for all $\alpha<\mu$. Finally, Fact \ref{fact: proj and complete emb} ensures that this is enough to infer the existence of a projection from $\mathbb{P}$ into $\ro(\mathbb{Q})$.
\end{proof}
The above proposition will be crucial in $\S\ref{sec: predistributive Forcings}$ to construct projections.

Finally, we record some preservation theorem under projections which are well-known:
\begin{fact}\label{projectionFacts}
    Suppose that $\pi:\mathbb{P}\to\Q$ is a projection. Then:
    \begin{enumerate}
        \item If $\P$ is $(\mu,\lambda)$-centered then $\Q$ is $(\mu,\lambda)$-centered.
        \item If $\P$ is $\lambda$-cc, then $\Q$ is $\lambda$-cc.
        \item If $\P$ is $(\kappa,\mu)$-distributive, then $\Q$ is $(\kappa,\mu)$-distributive.
    \end{enumerate}
\end{fact}
\begin{proof}
For $(1)$, If $\mathbb{P}=\bigcup_{\alpha<\lambda}\mathbb{P}_\alpha$ is a witnessing decomposition for $\mathbb{P}$ being $(\mu,\lambda)$-centered, let $\mathbb{Q}_\alpha$ be the upward closure of $\pi``\mathbb{P}_\alpha$, it is not hard to check that this is a witnessing decomposition witnessing that $\Q$ is $(\mu,\lambda)$-centered. 
    Both follow from their equivalent formulation using properties of the generic extension (see \cite{Bukovsk1973CharacterizationOG}) and Fact \ref{fact: equivalences for projections}.
\end{proof}
\subsection{Ultrafilters}

Here we discuss various orderings of filters and properties of filters that will be essential in our analysis in this paper.
We begin with the following key convention:
\begin{conv}
    All ultrafilters (not necessarily all filters!) considered in this paper are assumed to be $\sigma$-complete and uniform.
\end{conv}
For any ultrafilter $U$, we will denote the (transitive collapse of the) ultrapower of $V$ by $U$ as $M_U$ and the associated map as $j_U : V \to M_U$.

Furthermore we will make use of iterated ultrapowers.
For more details than we provide, see the excellent \cite{steel_iteratedultrapowers}.
The iterated ultrapower of $V$ by $U$ of length $\delta$ is a system $\{ M^U_\alpha, j^U_{\beta,\gamma} : \alpha < \delta \text{ and } \beta < \gamma < \delta \}$
where:
\begin{enumerate}
    \item $M^U_0 = V$.
    \item $M^U_{\alpha+1}$ is the ultrapower of $M^U_\alpha$ by $j^U_{0,\alpha}(U)$ and $j^U_{\alpha,\alpha+1}$ is the associated ultrapower map.
    \item $j^U_{\xi,\alpha+1} = j^U_{\alpha,\alpha+1} \circ j^U_{\xi,\alpha}$ for $\xi < \alpha$.
    \item Whenever $\lambda$ is limit, $M^U_\lambda$ is the direct limit of $\langle M^U_\alpha : \alpha < \lambda \rangle$ via the $j^U_{\alpha,\beta}$'s for $\alpha < \beta < \lambda$, and $j^U_{\alpha,\lambda}$'s are the direct limit maps.
\end{enumerate}
We remove the superscript $U$ when there is no ambiguity.

As for orderings of ultrafilters, we begin with the Rudin-Keisler ordering of ultrafilters:

\begin{definition}[Rudin-Keisler and Kat\v{e}tov Order]\label{def rk}
    Let $U,W$ be ultrafilters over sets $X,Y$ respectively.
    We define the \textit{Rudin-Keisler order} by $U \leq_{RK} W$ iff there is a function $f : Y \to X$ such that $A \in U$ iff $f^{-1}[A] \in W$.
    We write $U = f_* W$. 
    
    Given a filter $F$ over $X$.
    We define the \textit{Kat\v{e}tov order} by $F \leq_K W$, if there is an ultrafilter $F\subseteq U$ on $X$ such that $U\leq_{RK} W$.
\end{definition}

\begin{definition}
    Let $j : V \to M$ be an elementary embedding, $X$ any set, and $x \in X \cap M$.
    We derive an ultrafilter $U(x,j)$ over $X$ as follows: for any $A \subseteq X$,
    \[ A \in U(x,j) \iff x \in j(A) \]
\end{definition}

We use this notion to  characterize the Rudin-Keisler and Kat\v{e}tov order.
A proof of the following can be found in \cite{GoldbergUA}.
\begin{prop}
    Let $U,W$ be ultrafilters on $X,Y$ respectively.
    Then $U \leq_{RK} W$ via $f$ iff $U = U([f]_W,j_W)$. 
\end{prop}

\begin{cor}
    Let $F$ be a filter over $X$ and $W$ any ultrafilter.
    Then $F \leq_K W$ iff there is some  $x \in  j_W(X)$ such that $F \subseteq U(x,j_W)$.
\end{cor}
    We say that a filter $F$ (ultrafilter) over $X$ is \textit{$\omega$-Kat\v{e}tov below} $W$, denoted by $F \leq^\omega_K W$ $(F\leq^\omega_{RK} W)$ if there is an $x \in j^W_\omega(X)$ such that $F \subseteq U(x,i)$ ($F= U(x,i)$).
    
We will use the following lemma, whose proof can be found in \cite{GoldbergUA}.
\begin{lemma}
    Let $j : V \to M$ be an elementary embedding and let $U = U(x,j)$ for some $x \in M$.
    Then the map $k : M_U \to M$ where $k([f]_U) = j(f)(x)$ is an elementary embedding such that $k \circ j_U = j$.
\end{lemma}

We will also discuss products of ultrafilters.

\begin{definition}
    Let $U,W$ be ultrafilters over $X,Y$ respectively.
    The Fubini product $U \cdot W$ is the filter
    \[ \{ A \subseteq X \times Y : \{ x \in X : \{ y \in Y : (x,y) \in A\} \in W \} \in U \} \]
    For ultrafilters $U_0,\ldots,U_n$ we define $U_0 \cdot \ldots \cdot U_n$ by recursion by $U_0 \cdot (U_1 \cdot \ldots \cdot U_n)$.
\end{definition}
When we take the product of $U_i$ for $i < n$ and $U_i =  U$ for all $i < n$, we denote the product by $U^n$.
Products of ultrafilters give a combinatorial characterization of iterated ultrapowers.
\begin{prop}
    Let $\{ M_n, j_{m,n} : n \in \omega \}$ be the iterated ultrapower of $V$ by $U$ of length $\omega$.
    For each $n$, $M_n$ is the ultrapower of $V$ by $U^n$ and $j_{0,n}$ is the ultrapower map $j_{U^n}$.
\end{prop}
\begin{cor}
    Let $F$ be a filter and $U$ be an ultrafilter. $F\leq^\omega_{K}U$ iff there is $n<\omega$ such that $F\leq_{K} U^n$.
\end{cor}

To end this subsection, we define and discuss a canonical filter associated to any notion of forcing.
\begin{definition}
    Let $\Q$ be any notion of forcing and $q \in \Q$.
    We define
    \[ \mathcal{F}_q(\Q) \coloneq \{ D\subseteq \Q : D \text{ is open and dense below }q\} \]
    When $q = \one_\Q$ we omit the subscript.
\end{definition}

\begin{fact}\label{factProj}Let $\Q$ be a notion of forcing.
\begin{enumerate} 
    \item $\mathcal{F}_q(\mathbb{Q})$ is the filter  generated by $\mathcal{F}(\mathbb{Q})$ and $\mathbb{Q}/q:=\{p\in\mathbb{Q}: p\leq q\}$.
    \item $\mathbb{Q}$ is $\mu$-distributive if and only if $\mathcal{F}(\mathbb{Q})$ is $\mu$-complete.
    \item If $q_1 \leq q_2$, then $\mathcal{F}_{q_1}(\mathbb{Q})\subseteq \mathcal{F}_{q_2}(\mathbb{Q})$.
  
\end{enumerate}
\end{fact}
\begin{cor}\label{Cor: RK and distributive}    
    Let $U$ be a $\kappa$-complete ultrafilter. If $\mathcal{F}_{q}(\mathbb{Q})\leq^\omega_K U$ for every $q\in \mathbb{Q}$ then $\mathbb{Q}$ is $\kappa$-distributive.
\end{cor}
\begin{proof}
    In this case, for every $q\in\mathbb{Q}$, there is $x_q,n_q$ such that $\mathcal{F}_q(\mathbb{Q})\leq_K U(x_q,j_{0,n_q})$ for some $n<\omega$. Hence for every $q$, there is a $\kappa$-complete ultrafilter $\mathcal{W}_q$ such that $\mathcal{F}_{q}(\mathbb{Q})\subseteq \mathcal{W}_q$. To see that $\mathbb{Q}$ is $\kappa$-distributive, let $\langle D_\alpha:\alpha<\mu\rangle\subseteq \mathcal{F}(\mathbb{Q})$ for some $\mu<\kappa$. To see that $\bigcap_{\alpha<\kappa} D_\alpha$ is dense, let $q\in\mathbb{Q}$, and note that $D_\alpha\cap \mathbb{Q}/q\in \mathcal{F}_q(\mathbb{Q})\subseteq \mathcal{W}_q$ for all $\alpha<\mu$. Since $\mathcal{W}_q$ is $\kappa$-complete, $\bigcap_{\alpha<\mu}D_\alpha\cap \mathbb{Q}/q\in\mathcal{W}_q$ and in particular $\bigcap_{\alpha<\mu}D_\alpha\cap \mathbb{Q}/q\neq \emptyset$. We conclude that there is $p\leq q$, such that $p\in \bigcap_{\alpha<\mu}D_\alpha$, showing that $\bigcap_{\alpha<\mu}D_\alpha$ is dense.
\end{proof}
\begin{remark}
    Having $\mathcal{F}_q(\mathbb{Q})\leq^\omega_K U$ for every $q$ is not the same as having $\mathcal{F}(\mathbb{Q})\leq^\omega_K U$. A counterexample is obtained by taking $\mathbb{Q}$ which is the sum of two forcings $\mathbb{Q}_1$ and $\mathbb{Q}_2$, such that $\mathcal{F}(\mathbb{Q}_1)\leq^\omega_K U$ and $\mathcal{F}(\mathbb{Q}_2)\not\leq^\omega_K U$.
\end{remark}
In the case that $\mathcal{F}_q(\mathbb{Q})\leq_K U$ (namely without having to iterate $j_U$), the fact that $\mathcal{F}_q(\mathbb{Q})$ is the filter generated by $\mathcal{F}(\mathbb{Q})$ and $\mathbb{Q}/q$ gives the next simple proposition:
\begin{prop}\label{Prop: simple equivalence for Katetov of filter of dense sets}
    Let $\mathbb{Q}$ be a forcing notion and $U$ be a $\sigma$-complete ultrafilter. Then the following are equivalent:
    \begin{enumerate}
        \item for every $q\in\mathbb{Q}$, $\mathcal{F}_q(\mathbb{Q})\leq_K U$.
        \item $\bigcap j_U``\mathcal{F}(\mathbb{Q})$ is dense in $j_U``\mathbb{Q}$; namely every $q\in\mathbb{Q}$ there is $x\in \bigcap j_U``\mathcal{F}(\mathbb{Q})$ such that $x\leq j_U(q)$.
    \end{enumerate}
\end{prop}
Note that condition $(2)$ above trivially holds if $\bigcap j_U``\mathcal{F}(\mathbb{Q})$ is dense in $j_U(\mathbb{Q})$.
\subsection{Strongly Compact Prikry forcing}

Here we define the main forcing of interest in this paper, the strongly compact Prikry forcing.
First a brief reminder of some notions related to strong and supercompactness.
For cardinals $\kappa < \lambda$, we let $P_\kappa(\lambda) = \{ x\subseteq \lambda : |x| < \kappa \}$.
There is a natural order on $P_\kappa(\lambda)$: we set $x \prec y$ iff $x \subseteq y$ and $|x| < |\kappa \cap y|$.
An ultrafilter $\U$ over $\pkl$ is \emph{fine} if for all $\alpha < \lambda$, $\{ x \in \pkl : \alpha \in x \} \in \U$.
We say $\U$ is \emph{normal} if $\U$ is closed under diagonal intersections, that is whenever $A_\alpha \in U$ for $\alpha < \lambda$, we have $\Delta_\alpha A_\alpha \coloneq \{ x \in \pkl : x \in \bigcap_{\alpha \in x} A_\alpha \} \in \U$.

\begin{definition}
    $\kappa$ is $\lambda$-strongly compact if there is a $\kappa$-complete fine ultrafilter over $\pkl$.
    $\kappa$ is $\lambda$-supercompact if there is a $\kappa$-complete normal, fine ultrafilter over $\pkl$.
    $\kappa$ is strongly compact (resp. supercompact) if there is a $\kappa$-complete fine (resp. normal) ultrafilter over $\pkl$ for all $\lambda \geq \kappa$.
\end{definition}

Ultrapowers give us a useful characterization of strongly compact and supercompact ultrafilters.

\begin{prop}
    Let $\U$ be a $\sigma$-complete ultrafilter over $\pkl$.
    \begin{enumerate}
        \item $\U$ is fine iff there is an $A \in M_\U$ of size $<j(\kappa)$ such that $j``\lambda \subseteq A$.
        \item $\U$ is normal iff $j``\lambda = [\id]_\U$.
    \end{enumerate}
\end{prop}

Now we define the strongly compact Prikry forcing.
All facts about this forcing we mention can found in \cite{Gitik2010}.
\begin{definition}
    Fix a fine, $\kappa$-complete ultrafilter $\U$ over $P_\kappa(\lambda)$.
    The \emph{strongly compact Prikry forcing} $\P_\U$ consists of trees $T$ such that:
    \begin{enumerate}
        \item The elements of $T$ are finite, $\prec$-increasing sequences in $P_\kappa(\lambda)$.
        \item $T$ is a tree with the end-extension order.
        \item $T$ has a  \emph{stem} $s$, i.e. node $s \in T$ such that for all $t \in T$, either $t \leq_T s$ or $s \leq_T t$. 
        \item For every $t \geq_T s$, $\text{succ}_T(t) = \{ x \in P_\kappa(\lambda) : t {}^\smallfrown \langle x \rangle \in T \} \in \U$.
    \end{enumerate}
    For $T, S \in \P_\U$ we say $T \leq S$ iff $T \subseteq S$.
    If $t,s$ are the stems of $T,S$ respectively and $n = |t| - |s|$ we say $T$ is an \emph{$n$-step extension} of $S$.
    If $n = 0$ (so $T$ and $S$ share the same stem) we say $T$ is a \emph{direct extension} $S$, and write $T \leq^* S$. 
\end{definition}
\begin{notation}
    For $T\in \mathbb{P}_{\mathcal{U}}$, we will usually write $(t,T_t)$, where $t$ is the stem of $T$ and $T_t$ is the tree above $T$ (i.e. $T_s=\{s\setminus t\mid s\in T, \ t\prec s\}$). We also let the $n^{\text{th}}$ level of $T$, denoted by $\mathcal{L}_n(T)$, be $T\cap P_\kappa(\lambda)^n$. 
\end{notation}
Note that $\mathcal{L}_n(T_t)$ form all possible $n$-step extensions of $t$. The following fact follows immediately from the definition of $\U^n$.
\begin{fact}
    Let $\U$ be a fine $\kappa$-complete ultrafilter over $\pkl$. Suppose that $(t,T)\in \mathbb{P}_{\U}$, then for every $n$, $\mathcal{L}_n(T)\in \U^n$. In the other direction, if $A\in \mathcal{U}^n$, then there is a stem $t$ and a tree $T$ such that $\mathcal{L}_n(T)\subseteq A$ and $(t,T)\in \P_\U$.
\end{fact}
We note that in the case that $\U$ is normal, there is a dense subset of $\P_\U$ where $\text{succ}_T(t)$ is the same for all splitting nodes $t$.
In fact, every condition has a direct extension in this dense subset.
Hence we may identify this dense subset with the partial order of pairs $(s,A)$ where $s$ is a finite, $\prec$-increasing sequence of elements of $\pkl$ and $A \in \U$.
This is the usual presentation of supercompact Prikry forcing, and we will call this partial order as such.
As a convention, when $\U$ is normal we will always consider only conditions in this dense subset and treat the conditions as pairs $(s,A)$ instead of trees.

The usual Tree Prikry forcing and Prikry forcing with a normal measure can be viewed as a particular case of the strongly compact and supercompact Prikry forcings in the case $\kappa=\lambda$. For a $\kappa$-complete ultrafilter $U$ over $\kappa$ we denote by $Pr(U)$ the classical Prikry forcing from \cite{Prikry}.

A generic $G$ yields a sequence $\langle x_n  : n\in \omega \rangle \subseteq \pkl$, and by genericity $\bigcup_n x_n = \lambda$.
So $\P_\U$ singularizes $\lambda$, and in fact collapses all cardinals in the interval $(\kappa,\lambda]$ to have cardinality $\kappa$.
Since any conditions with the same stem are compatible, it's easy to see that $\P_\U$ is $\lambda^{<\kappa}$-cc and so preserves all cardinals above $\lambda^{<\kappa}$.
$\kappa$ itself is preserved and singularized by the sequence $\langle x_n \cap \kappa : n \in \omega \rangle$.
The preservation of $\kappa$ follows from the fact that no bounded subsets of $\kappa$ are added, which follows from the following property.

\begin{lemma}(Prikry Property)
    Let $\varphi$ be any statement of the forcing language and $T \in \P_\U$.
    There is an $S \leq^* T$ such that $S \parallel \varphi$.
\end{lemma}

There is a strengthening of the Prikry property we will use:

\begin{lemma}(Strong Prikry Property)
    Let $D \subseteq \P_\U$ be open dense and $T \in \P_\U$.
    There is a $S \leq^* T$ and $n \in \omega$ such that every $n$-step extension of $S$ in $D$.
\end{lemma}

In general, we can identify the Prikry sequence $\langle x_n : n \in \omega \rangle$ and recover one from the other.
Thus it makes sense to discuss $\P_\U$-generic sequences, not only filters.
We have the following characterization of generic sequences which appears as Theorem $4.3$ in \cite{Hamkins1997}.
\begin{lemma}
    Let $V \subseteq W$ be two models of ZFC with $\langle x_n : n\in \omega \rangle \in W$.
    Then $\langle x_n : n\in \omega \rangle$ is $\P_\U$-generic over $V$ iff for all $F : {}^{<\omega}P_\kappa(\lambda) \to \U$ in $V$, there is an $m$ such that for all $n > m$ we have $x_n \in F( \langle x_k : k < n\rangle)$.
\end{lemma}
Recall in the case that $\U$ is normal we treat conditions as pairs $(s,A)$, yielding a simpler characterization of Prikry-generic sequences: $\langle x_n : n \in \omega \rangle$ is $\P_\U$-generic iff for all $A \in \U$ there is an $m \in \omega$ such that $x_n \in A$ for all $n > m$.

Finally, we will also use the Bukovsky-Dehornoy analysis of Prikry forcing and iterated ultrapowers.
A proof of this more general form can be found in \cite{Hamkins1997}.
\begin{theorem}\label{thm: general bukovsky-dehornoy}
    Let $\mathcal{U}$ be a fine ultrafilter over $\pkl$.
    Let $\{ M_n , j_{m,n} : m < n \in \omega\}$ be the iterated ultrapower of $V$ by $\U$ of length $\omega$.
    Consider the sequence
    \[ S = \{ j_{n+1,\omega}([\id]_{\U^n}) : n < \omega \} \]
    Then $S$ is $j_{0,\omega}(\P_\U)$-generic over $M_\omega$.

    Furthermore, if $X = \{ x_n : n \in \omega \}$ is a sequence $j_{0,\omega}(\P_\U)$-generic over $M_\omega$ then for all but finitely many $n$, $\U = U(x_n, j_{0,\omega})$.
\end{theorem}

\section{Intermediate models of the strongly compact forcing}\label{section: RK}
In this section we study the connection between the Rudin-Keisler/Kat\v{e}tov predecessors of the ultrafilters used in a Prikry forcing and subforcings of that Prikry forcing. 
We start with the class of $\sigma$-distributive forcings:
\begin{prop}\label{fact: dis sub forcing}
    Let $\mathcal{U}$ be a $\kappa$-complete fine ultrafilter on $P_\kappa(\lambda)$, where $\lambda\geq\kappa$. If  $\mathbb{Q}$ is a projection of $\mathbb{P}_{\mathcal{U}}$, then:
    \begin{enumerate}
    \item $\Q$ is $(\kappa,\lambda)$-centered.
        \item If $\Q$ is $\sigma$-distributive then  $\mathbb{Q}$ is $\kappa$-distributive.
    \end{enumerate}
\end{prop}
\begin{proof}
$(1)$ follows from $\P_{\mathcal{U}}$ being $(\kappa,\lambda)$-linked and \ref{projectionFacts}. For $(2)$,
    see \cite[Proposition 12]{Benhamou_Hayut_Gitik}.
\end{proof}
The following theorem is about $\kappa$-distributive forcings (by Corollary \ref{Cor: RK and distributive}):
\begin{theorem}\label{thm: seed implies weak projection}
    Suppose that $\mathbb{Q}$ is a forcing notion such that for every $q\in\mathbb{Q}$, $\mathcal{F}_q(\mathbb{Q})\leq^\omega_{K} \mathcal{U}$. Then there is a weak projection from $\mathbb{P}_\mathcal{U}$ to $\mathbb{Q}$.
\end{theorem}
\begin{proof}
    Let $\mathcal{U}$ be a fine $\kappa$-complete ultrafilter on $P_\kappa(\lambda)$. For every $q$, let $n=n_q$ be minimal such that $\mathcal{F}_q(\mathbb{Q})\leq_K \mathcal{U}^n$ and let $g_q:[P_\kappa(\lambda)]^n\to \mathbb{Q}$ witness that. Then:
    \begin{enumerate}
        \item $\mathcal{F}_q(\mathbb{Q})\subseteq (g_{q})_*(\mathcal{U}^n)$.
        \item We may assume that $\rng(g_q)\subseteq \mathbb{Q}/q$.
    \end{enumerate} 
    First we note that if $q_1\geq q_2$ then $n_{q_1}\leq n_{q_2}$; this follows from the minimality and Fact \ref{factProj}(3). Now define for each $p=(x_1,x_2,...,x_n,T)$ its projection as follows: If $n<n_{\one_{\mathbb{Q}}}$, set $\pi(p)=\one_{\mathbb{Q}}$. Generally, we define recursively a decreasing sequence of conditions $q^p_k$ and non decreasing numbers $0\leq m_k\leq n$ for $0\leq k\leq n$. At the end we will set $\pi(p)=q^p_n$. We set $q^p_0=\one_{\mathbb{Q}}$ and $m_0=0$. Suppose that $q^p_k, m_k$ were defined for $1\leq k<n$ and let us define $q^p_{k+1},m_{k+1}$. Consider $n_{q^p_k}$. If $k<n_{q^p_k}+m_k$ set $q^p_{k+1}=q^p_k$ and $m_{k+1}=m_{k}$ and note that $n_{q^p_k}=n_{q^p_{k+1}}$ (hence the sequence is piecewise constant). Otherwise, if $k=n_{q^p_k}+m_k$, then $(x_{m_k+1},...,x_{k})$ is an $n_{q^p_k}$-tuple, and thus we can set
    $$q^p_{k+1}:=g_{q^p_k}(x_{m_k+1},...,x_{k})\text{, and }m_{k+1}:=n_{q^p_k}+m_k.$$ Then by $(2)$, $q^p_{k+1}\leq q^p_k$. We say that $p$ has no residue if $n=m_n$.
    
    Note that $q^p_1\geq q^p_2\geq...\geq q^p_n$. Also note that if $q=(t,T)\geq (s,S)=p$, then since $s\sqsubseteq t$, we will have that $(q_1^p,...,q_{|s|}^p)\sqsubseteq (q_1^q,...,q_{|t|}^q)$. Hence $\pi(p)\geq \pi(q)$. 
    Now suppose that $p=(t,T)\in \mathbb{P}$ and $D$ is dense open in $\mathbb{Q}$. 
    Extend $p$ if necessary so that $p$ has no residue. 
    Then by $(1)$, $g_{\pi(p)}^{-1}[D]\cap \mathbb{Q}/{\pi(p)}\in \mathcal{U}^{n_{\pi(p)}}$. Find any $s\in \mathcal{L}_{n_{\pi(p)}}(T)\cap g_{\pi(p)}^{-1}[D]\cap \mathbb{Q}/{\pi(p)}$ and consider $(t^\smallfrown s,T_s)\leq p.$ Then, since $p$ has no residue, and by definition of $\pi$, $\pi((t^\smallfrown s,T_ss))=g_{\pi(p)}(s)\in D$. It follows that $\pi^{-1}[D]$ is dense in $\mathbb{P}$, hence by Fact \ref{fact: equivalences for projections}, $\pi$ is a weak projection.
\end{proof}
 Note that the projection defined in the previous theorem depends only on the stem of the condition. This will be addressed again in the Section \ref{sec: Tukey type}. \begin{remark}\label{Remark: full projection}
    If $\mathbb{Q}$ has a dense subset of cardinality $\lambda$, then by changing  $g_{\one_{\mathbb{Q}}}$ on a null set of cardinality $\lambda$, we can ensure that $\rng(g_{\one_{\mathbb{Q}}})$ is dense. This ensures that $\rng(\pi)$ is dense, which implies (by Fact \ref{fact: equivalences for projections}) that $\pi$ is in fact a full projection. In general, the previous theorem cannot be improved to a full projection, for example an extremely large disjoint sum of copies of the trivial single atom forcing will have $\mathcal{F}_q(\mathbb{Q})\leq^\omega_K\mathcal{U}$ but cannot be a projection of $\mathbb{P}_{\mathcal{U}}$ since it has an extremely large chain condition.
\end{remark}
\begin{definition}\label{def: pregenerated} We say that $\mathbb{Q}$ is $\lambda$-pregenerated if $\mathcal{F}(\mathbb{Q})$ is $\lambda$-generated.
\end{definition}

\begin{fact}
    If $\mathbb{Q}$ is $\mu$-cc then $\mathbb{Q}$ is $|\mathbb{Q}|^{<\mu}$-pregenerated. 
\end{fact}
\begin{proof}
    Every dense set contains a maximal antichain which is a subset of $\mathbb{Q}$ of size ${<}\mu$.
\end{proof}
\begin{cor}[Folklore]\label{cor: folk}
Let $\mathcal{U}$ be a fine measure over $P_\kappa(\lambda)$.
Every $\kappa$-distributive forcing which is $\lambda$-pregenerated is a weak projection of $\P_{\mathcal{U}}$.
\end{cor}
\begin{proof}
    Let $\mathbb{Q}$ be a $\kappa$-distributive $\lambda$-pregenerated forcing. Let $\mathcal{U}$ be a strongly compact measure on $P_\kappa(\lambda)$, and let $j_{\mathcal{U}}:V\to M_{\mathcal{U}}$ be the ultrapower embedding. 
    By $\lambda$-strong compactness, there is $x\in M_{\mathcal{U}}$, $M_{\U}\models|x|<j_{\mathcal{U}}(\kappa)$ such that  $j_{\mathcal{U}}``\lambda\subseteq x$. 
    By definition, there is a generating family $\mathcal{C}$  for $\mathcal{F}(\mathbb{Q})$ of size $\lambda$. Via appropriate coding, we can find $x^*\in M_{\mathcal{U}}$ of size ${<}j_{\mathcal{U}}(\kappa)$, $j_{\mathcal{U}}``\mathcal{C}\subseteq x^*\subseteq \mathcal{F}(j_{\mathcal{U}}(\mathbb{Q}))$. 
    By Fact \ref{factProj}, since $\mathbb{Q}$ is $\kappa$-distributive, $\mathcal{F}(\mathbb{Q})$ is a $\kappa$-complete. 
    Hence, by elementarity, $\mathcal{F}(j_{\mathcal{U}}(\mathbb{Q}))$ is $j_{\mathcal{U}}(\kappa)$-complete in $M_{\mathcal{U}}$. 
    It follows that $\bigcap x^*\in \mathcal{F}(j_{\mathcal{U}}(\mathbb{Q}))$ is dense and also $\bigcap j_\U``\mathcal{F}(\mathbb{Q})$ is dense in $j_\U``\mathbb{Q}$. By Proposition \ref{Prop: simple equivalence for Katetov of filter of dense sets}, we conclude that for every $q\in\mathbb{Q}$, $\mathcal{F}_q(\mathbb{Q})\leq_K {\mathcal{U}}$. 
    By Theorem \ref{thm: seed implies weak projection}, $\mathbb{P}_{\mathcal{U}}$ weakly projects to $\mathbb{Q}$.
\end{proof}
Similar to Remark \ref{Remark: full projection}, the previous corollary gives a full projection in case the forcing $\mathbb{Q}$ has a dense subset of cardinality at most $\lambda$.
\begin{cor}\label{cor: distibutive consistency}
    Suppose that $\kappa$ is $\lambda$-supercompact and $\mu\leq\lambda$ is a cardinal. Then for any $\kappa$-distributive forcing $\mathbb{Q}$ of cardinality $\mu$ which is $\lambda$-pregenerated, there is a $\mu$-strongly compact measure $\mathcal{U}$ such that $\mathbb{P}_{\mathcal{U}}$ projects onto $\mathbb{Q}$.
\end{cor}
\begin{proof}
  Since $|\mathbb{Q}|=\mu$, let us assume without loss of generality that $\mathbb{Q}=\mu$. Of course we may assume that $\mu\geq\kappa$ (by the $\kappa$-distributivity), and we may also assume that for every dense set $D\in \mathcal{F}_q(\mathbb{Q})$, $|D|=\mu$ (otherwise we move to the minimal $\mu'$ such that there is a dense subset of $\mathbb{Q}$ of cardinality $\mu$). Let $j:V\to M$ be a $\lambda$-supercompact embedding. By the $\lambda$-pregenerated assumption, let $\mathcal{C}_q\subseteq \mathcal{F}_q(\mathbb{Q})$ be generating of size $\lambda$. Hence $j``\mathcal{C}_q\in M$. By $\kappa$-distributivity (and $j(\kappa)$-distributivity of $j(\mathbb{Q})$),  there is $x_q\in \bigcap j``\mathcal{C}_q\subseteq \bigcap j``\mathcal{F}_q(\mathbb{Q})$. Notice that since every dense open set of $\mathbb{Q}$ has cardinality $\mu$, we can choose that sequence $x_q$ increasing above $\sup j``\mu$. Consider $\vec{x}=\{x_q: q\in\mu\}$ and notice that since $\mu\leq\lambda$, and since $M^{\leq\lambda}\subseteq M$, $\vec{x}\in M$. We can now derive a measure ${\mathcal{U}}$ from $j``\mu\cup\vec{x}$ (this has order type $\mu+\mu$). Notice that $j``\mu\cup \vec{x}\in P_{j(\kappa)}(j(\mu))$, and therefore ${\mathcal{U}}$ is a fine $\kappa$-complete measure on $P_\kappa(\mu)$.
   There is a factor map $k:M_{\mathcal{U}}\to M$ defined by $k([f]_{\mathcal{U}})=j(f)(j``\mu\cup\vec{x})$.
   We have that $k([id]_{\mathcal{U}})=j``\mu\cup\vec{x}$, and $[id]_{\mathcal{U}}\cap\sup( j_{\mathcal{U}}``\mu)=j_{\mathcal{U}}``\mu$, which means that $M_{\mathcal{U}}$ is closed under $\mu$-sequences (see e.g. \cite{GoldbergUA}). We have that for every $i<\mu$, the $\mu+i$ element of $[id]_{\mathcal{U}}$ is mapped under $k$ to the $\mu+i$ element of $j``\mu\cup\vec{x}$ which is just $x_i$. Hence $x_i=k(y_i)$ for some $y_i$. This is enough to show that $y_q\in \bigcap j_{\mathcal{U}}`` \mathcal{F}_q(\mathbb{Q})$. Finally, we conclude that $\mathcal{F}_q(\mathbb{Q})\leq_K {\mathcal{U}}$ and use Theorem \ref{thm: seed implies weak projection} (and Remark \ref{Remark: full projection}).   
\end{proof}
To see a concrete corollary, we have the following:
\begin{cor}
    Suppose that $\kappa$ is $\kappa^{++}$-supercompact, $2^\kappa=\kappa^+$ and $2^{\kappa^+}=\kappa^{++}$. 
    Then $\col(\kappa,\kappa^+)$ and $\add(\kappa^+,1)$  are projections of $\mathbb{P}_{\mathcal{U}}$ for some $\kappa^+$-strongly compact measure ${\mathcal{U}}$ over $P_\kappa(\kappa^+)$.
\end{cor}
Without a small generating set for the filter of dense open sets, we can say the following:
\begin{theorem}\label{thm: omega generated from projection}
    Suppose $|\mathbb{Q}| \leq \lambda$ and ${\mathcal{U}}$ is a fine $\kappa$-complete ultrafilter on $P_\kappa(\lambda)$.
    Suppose further that  
        $\pi : \mathbb{P}_{\mathcal{U}} \to \mathbb{Q}$ is a weak projection. Then \begin{align*}\one\Vdash_{\mathbb{P}_\U}``\pi``\dot{G}\text{ is generated by countably many elements}".
    \end{align*} 
    \end{theorem}
    \begin{proof}
    We may assume that $\mathbb{Q}=\lambda$.
    Let $\langle j_{n,m} \colon M_n \to M_m : n < m \leq \omega \rangle$ be the iterated ultrapower of $V$ by $U$.
    Let $c_n = [\id]_{j_{0,n}({\mathcal{U}})}$. Let $p\in\mathbb{P}_U$, we will show that there is a $M_{\omega}$-generic filter $G_\omega\subseteq j_\omega(\mathbb{P})$ which includes $j_\omega(p)$ such that $j_{\omega}(\pi)``G_\omega$ is generated by a countable set in $M_\omega[G_\omega]$. Then we use elementarity to conclude that $p\not\Vdash_{\P_\U} ``\pi``\dot{G}$ is not countably generated", and since $p$ was arbitrary, the lemma follows. 
    
    Suppose that $p=(t_0,T_0)$, and note that for each $n<\omega$, $$(j_\omega(c_0),...,j_\omega(c_{n-1}))\in \mathcal{L}_n(j_\omega(T_0)).$$ Letting $G_\omega$ be the filter generated by the sequence $t_0{}^\smallfrown (j_\omega(c_0),j_\omega(c_{1}),...)$, we see that $j_\omega(p)\in G_\omega$ (this is generic by Theorem \ref{thm: general bukovsky-dehornoy}). Suppose without loss of generality that $t_0=\emptyset$, and recall that $M_{\omega}[G_\omega]=\bigcap_{n<\omega}M_n$. Let $H_\omega\subseteq j_{\omega}(\mathbb{Q})$ be the  $M_\omega$-filter generic over $M_\omega$ generated by the projection of $G_\omega$, namely, $H_\omega=j_{\omega}(\pi)``G_{\omega}$. We will construct a decreasing sequence of conditions $\langle q_n : n < \omega \rangle\subseteq j_{0,\omega}(\mathbb{Q})$ which generates $H_{\omega}$.

    In $M_n$ we can define 
    \[F_n = \{ q \in j_{0,n}(\mathbb{Q}) : \exists T \, j_{0,n}(\pi)(\langle c_0,\ldots,c_{n-1} , T\rangle) \leq q \}\subseteq j_{0,n}(\lambda)\]
By strong compactness there is  $x_n\in M_{n+1}$ such that $|x_n|<j_{0,n+1}(\kappa)$ such that $j_{n,n+1}``F_n\subseteq x_n$. Also note that $j_{n+1,\omega}^{-1}[H_\omega]\in M_{n+1}$, hence $$Y_n:=x_n\cap j_{n+1,\omega}^{-1}[H_\omega]\in M_{n+1}.$$   Since the critical point of $j_{n+1,\omega}$ is $j_{0,n+1}(\kappa)>|Y_n|$, we have that $$Y^*_n:=j_{n+1,\omega}``Y_n=j_{n+1,\omega}(Y_n)\in M_\omega.$$ Since $H_{\omega}$ is the projection of $G_{\omega}$, $j_{n,\omega}``F_n\subseteq H_\omega$, hence $$j_{n,\omega}``F_n= j_{n+1,\omega}``j_{n,n+1}``F_n\subseteq j_{n+1,\omega}``Y_n=Y_n^*.$$

   Now, working in $M_\omega[H_\omega]$, we have the sequence $\langle Y^*_n: n<\omega\rangle$, and we can pick inductively a decreasing sequence $q_0\geq q_1\geq q_2\dots$ in $H_\omega$, such that $q_{n} \Vdash_{j_{0,\omega}(\mathbb{Q})} \check{Y}^*_n \subseteq \dot{H}_\omega$.
Then $\{q_n: n<\omega\}$ generates $H_\omega$, indeed, if $q\in H_\omega$ then $j_{0,\omega}(\pi)(p)\leq q$ for some $p\in G_{\omega}$. We can then find $n$ large enough so that $p=j_{n,\omega}(\langle c_0,...,c_{n-1},T\rangle)$ for some $T$ which is $j_{0,n}(\mathcal{U})$-fat tree and $q=j_{n,\omega}(q')$ for some $q'\in j_{0,n}(\mathbb{Q})$. By elementarity, $q'\geq j_{0,n}(\pi)(\langle c_0,...,c_{n-1},T\rangle)$ and therefore $q'\in F_n$ and $q\in j_{n,\omega}``F_n$. Hence $q_n\leq q$.
    \end{proof}
\begin{remark}\label{Remark following the RK-construction}
{} \ 
\begin{enumerate}
    \item Even without assuming $|\mathbb{Q}|\leq\lambda$ we always decompose the generic for $j_{0,\omega}(\mathbb{Q})$ inside $M_{\omega}[G_{\omega}]$ as a union of countably many pieces, but these pieces may not be in the ground model $M_\omega$ and we may not be able to create a generic from that decomposition.
    \item The previous proof can be performed locally, namely, if $q\in \mathbb{Q}$ is a condition such that for some $p\in \mathbb{P}_{\mathcal{U}}$, $\pi(p)\leq q$, then we can ensure that $j_{0,\omega}(q)\in H_\omega$. In particular, if $\pi$ is a projection (i.e. $\rng(\pi)$ is dense), then for every condition $q\in\mathbb{Q}$ we can construct a generic $H_\omega$ which will include $j_{0,\omega}(q)$. 
    \item Following \cite[Thm.13]{Benhamou_Hayut_Gitik} and its succeeding remark, if $\mathbb{Q}$ is $\kappa$-distributive, then we can find $n<\omega$ such that some condition $q^*\in j_{0,n}(\mathbb{Q})$ will enter every set of the form $j_{0,n}(D)$, where $D\in \mathcal{F}(\mathbb{Q})$. This in particular implies that $\mathcal{F}(\mathbb{Q})\leq_K^\omega \mathcal{U}$.
    \item Following $(2),(3)$, if there is a projection of $\mathbb{P}_{\mathcal{U}}$ to $\mathbb{Q}$, then for every $q\in\mathbb{Q}$, $\mathcal{F}_q(\mathbb{Q})\leq_K^\omega \mathcal{U}$.
\end{enumerate}     
\end{remark}

\begin{cor}\label{corollary: Katetov dist}
     Let $\mathcal{U}$ be a fine $\kappa$-complete  measure over $P_\kappa(\lambda)$, and let $\mathbb{Q}$ be a forcing of cardinality $\leq\lambda$. Then the following are equivalent:
     \begin{enumerate}
         \item For every $q\in\mathbb{Q}$, $\mathcal{F}_q(\mathbb{Q})\leq^\omega_K \mathcal{U}$.
         \item $\mathbb{Q}$ is a  $\kappa$-distributive projection of $\mathbb{P}_\mathcal{U}$.
     \end{enumerate}
\end{cor}
What does being a projection mean for non $\kappa$-distributive forcings of cardinality $\lambda$ in terms of the Rudin-Keisler order? For example, $\mathcal{F}(Pr(U))$?

    Let $\fin$ denote the ideal of finite subsets of $\omega$. We say that $U\leq^{\fin}_{RK} W$ if there is a sequence of ultrafilters $W_n\leq_{RK} W^{m_n}$ such that $U=(\fin)^*\text{-}\lim W_n$. Given a filter $F$, we denote by  $F\leq_K^{\fin} W$ if $F$ extends to an ultrafilter $U$ such that  $U\leq^{\fin}_{RK}W$.
Given a sequence $x_1,x_2,...\in M^W_\omega$, and $X\in V$ is such that for every $i$, $x_i\in j^W_\omega(X)$, we define 
${\FIN}(j_\omega,\vec{x})=F$ over $X$ by $A\in F$ iff $\exists N\forall n\geq N$, $x_n\in j^W_{0,\omega}(A)$.
\begin{prop}
    $U\leq^{\fin}_{RK}W$ iff there are $x_0,x_1,...x_n,...\in M^W_{\omega}$ such that $U={\FIN}(j^W_\omega,\vec{x})$. 
\end{prop}
\begin{proof}
    If $W_n\leq_{RK}W^{m_n}$, then there is $y_n\in M^W_{m_n}$ such that $A\in W_n$ iff $y_n\in j^W_{0,m_n}(A)$. Let $x_n=j^W_{m_n,\omega}(y_n)$. Then $A\in W_n$ iff $x_n\in j^W_{0,\omega}(A)$. Hence
    $A\in (\fin)^*\text{-}\lim W_n$ iff $\exists N\forall n\geq N$, $A\in W_n$ namely $j^W_{0,\omega}(A)$ contains a tail of the $x_n$'s. In the other direction, if $U={\FIN}(j^W_\omega,\vec{x})$, then for each $n$, there is $m_n$ such that $x_n=j^W_{m_n,\omega}(y_n)$ for some $y_n\in M_{m_n}$. Let $W_n=U(j^W_{0,m_n},y_n)$. Then $W_n\leq_{RK} W^{m_n}$ and $A\in (\fin)^*\text{-}\lim W_n$ if and only if for a tail of $n$'s $y_n\in j^W_{0,m_n}(A)$ which is if and only if $x_n\in j^W_{0,\omega}(A)$. Hence $U=(\fin)^*\text{-}\lim W_n$.
\end{proof}
Given a sequence of ultrafilters over $\mathbb{Q}$, $U_0,U_1,...,U_{n-1}$, we say that $(U_0,...,U_{n-1})$ is an \textit{adequate sequence} if $$\{(q_0,...,q_{n-1})\mid q_0\geq q_1\geq ...\geq q_{n-1}\}\in U_0\cdot U_1\cdot...\cdot U_{n-1}.$$ A sequence $U_0,U_1,U_2,..$ is adequate if for each $n$, $(U_0,..,U_n)$ is adequate.
\begin{definition}\label{Def: strong RK} Given a forcing $\mathbb{Q}$, we say that $U$ on $\mathbb{Q}$ is \textit{strongly $\fin$-RK} below $W$, denoted by $U\leq^{\fin}_{sRK} W$ if there is an adequate sequence $W_n\leq_{RK} W$, such that $U=\fin\text{-}\lim W_n$. Similarly, we say that $F\leq^{\fin}_{sK}W$ if there is $U\leq^{\fin}_{sK}W$ such that $F\subseteq U$. 
\end{definition}
The proof of the following lemma is immediate:
\begin{lemma}\label{lemma for  strongRK}
    $U\leq^{\fin}_{sRK}W$ iff there is decreasing sequence $q_0\geq q_1\geq...\in j_{0,\omega}(\mathbb{Q})$ such that $U={\FIN}(j^W_\omega,\vec{x})$.
\end{lemma} \qed

We are now ready to prove the equivalent condition to the existence of a weak projection:
\begin{theorem}\label{Thm: more general equivalence}
    Let $\mathbb{Q}$ be a forcing notion with $|\mathbb{Q}|\leq\lambda$ and $\mathcal{U}$ is a fine $\kappa$-complete ultrafilter on $P_\kappa(\lambda)$. The following are equivalent
    \begin{enumerate}
        \item [(I)] $\ro(\mathbb{Q})$ is a weak projection of $\mathbb{P}_\mathcal{U}$.
        \item [(II)] $\mathcal{F}(\mathbb{Q})\leq^{\fin}_{sK} \mathcal{U}$.
    \end{enumerate}
\end{theorem}
Before the proof of the theorem we need the following Lemma for iterated ultrapowers:
\begin{lemma}\label{Lemma for iteration}
    Suppose that $U$ is a $\sigma$-complete ultrafilter in some model $M_0$, and let $(j^U_{n,m}:M_n\to M_m: n\leq m\leq\omega)$ be the $\omega^{\text{th}}$-iteration of $U$.
    \begin{enumerate}
        \item For any $n,m$, $j^U_{0,n}(j^U_{0,m})=j^U_{n,n+m}$.
        \item For every $n$, $j^U_{0,n}(j^U_{0,\omega})=j^U_{n,\omega}$.
        \item For every $m\leq n$, $j^U_{m,n}(j_{m,\omega})=j^U_{n,\omega}$.
        \item If $m\leq n$, then for any $x\in M_m$, $j_{m,n}(j_{m,\omega}(x))=j_{m,\omega}(x)$.
    \end{enumerate}
\end{lemma}
\begin{proof}[Proof of Lemma.]
    By elementarity, $j^U_{0,n}(j^U_{0,m})=j^{j^U_{0,n}(U)}_{0,m}$ is defined to be the $m^{\text{th}}$-iteration of $M_n$ by $j^U_{0,n}(U)$. Namely, the iterated ultrapower by the measures $j^U_{0,n}(U),j^U_{0,n+1}(U),...,j^U_{0,n+(m-1)}(U)$ which is exactly the continuation of the iteration of $U$ from $M_n$ to $M_{n+m}$. Item $(2)$ follows immediately from $(1)$. $(3)$ follows by applying $(2)$ in the model $M_m$. For $(4)$, we use  item $(3)$
    $$j^U_{m,n}(j^U_{m,\omega}(x))=j^U_{m,n}(j^U_{m,\omega})(j_{m,n}(x))=j^U_{n,\omega}(j^U_{m,n}(x))=j^U_{m,\omega}(x)$$
\end{proof}
\begin{proof}[Proof of Thm \ref{Thm: more general equivalence}.]
    Assume $(I)$, then by Theorem \ref{thm: omega generated from projection}, there is $q_0\geq q_1...$ elements of $j_\omega(\mathbb{Q})$ which generates an $M_\omega$-generic filter $H_\omega$. We claim that $\mathcal{F}(\mathbb{Q})\subseteq {\FIN}(j_\omega,\vec{q})$ which by Lemma \ref{lemma for  strongRK} implies that $\mathcal{F}(\mathbb{Q})\leq^{\fin}_{sK} {\mathcal{U}}$. Indeed, if $D\in \mathcal{F}(\mathbb{Q})$, then $j_\omega(D)\in M_\omega$ is dense open and therefore $j_{\omega}(D)\cap H_\omega\neq\emptyset$. Since the $q_n$'s are generating, there is $N$ such that $q_N\in j_{\omega}(D)\cap H_\omega$, and since $D$ is open, for all $n\geq N$, $q_n\in j_{\omega}(D)$. By definition, this means that $D\in {\FIN}(j_\omega,\vec{q})$. 
    
    In the other direction, assume $(II)$, namely that  $\mathcal{F}(\mathbb{Q})\leq^{\fin}_{sK} {\mathcal{U}}$, and let $\langle q_n\mid n<\omega\rangle\subseteq j_{\omega}(\mathbb{Q})$ be a decreasing sequence such that $\mathcal{F}(\mathbb{Q})\subseteq {\FIN}(j_\omega,\vec{q})$. By adding $\one_{\mathbb{Q}}$ to the sequence, if necessary, we may assume without loss of generality that $q_0=\one_{j_\omega(\mathbb{Q})}$. Let $p_n$ and $m_n$ be such that $p_n\in M_{m_n}$ and $j_{m_n,\omega}(p_n)=q_n$ (hence $p_0=\one_{\mathbb{Q}}$ and $m_0=0$). In $V$, consider $\vec{q}{}^0=\langle q_n\mid n<\omega\rangle$. Let $\vec{q}{}^1:=\langle j_{0,m_1}(q_n)\mid n<\omega\rangle$. By elementarity of $j_{0,m_1}$, $\vec{q}{}^1$ diagonalizes $j_{0,m_1}(j_{0,\omega})(D)$ for every $D\in j_{0,m_1}(\mathcal{F}(\mathbb{Q}))$. By Lemma \ref{Lemma for iteration}(2), $j_{0,m_1}(j_{0,\omega})=j_{m_1,\omega}$, hence $\vec{q}{}^1$ diagonalizes $j_{m_1,\omega}(D)$ for every $D\in \mathcal{F}(j_{0,m_1}(\mathbb{Q}))$.  
We would like to note here that by $(4)$,
    $$q^1_0=j_{0,m_1}(q^0_0)=j_{0,m_1}(j_{0,\omega}(p_0)))=j_{0,\omega}(p_0)=q^0_0.$$
    Next let $$\vec{q}{}^2=\langle j_{m_1,m_2}(q^1_n)\mid n<\omega\rangle=\langle j_{0,m_2}(q^0_n)\mid n<\omega\rangle.$$ By elementarity, for every $D\in j_{0,m_2}(\mathcal{F}(\mathbb{Q}))$, $j_{m_2,\omega}(D)$ is diagonalized by $\vec{q}{}^2$. Also, for $i<2$ we have
    $$q^2_i=j_{0,m_2}(q^0_i)=j_{0,m_i}(j_{m_i,m_2}(j_{m_i,\omega}(p_i))=j_{0,m_i}(j_{m_i,\omega}(p_i))=j_{0,m_i}(q^0_i)=q^i_i$$

    We define $\vec{q}{}^n$ similarly. Set $q^*_n=\vec{q}_n{}^n(=\vec{q}_n{}^{n+1}=\vec{q}_n{}^{n+2}=\dots)$. Let us show that $\vec{q}{}^*$ diagonalizes $j_{0,\omega}(\mathcal{F}(\mathbb{Q}))$, namely the filter generated by $\vec{q}^*$ is $M_\omega$-generic. Let $D\in j_{0,\omega}(\mathcal{F}(\mathbb{Q}))$, then $D=j_{m_n,\omega}(D_0)$ for some $n$ large enough. Then by the property of $\vec{q}{}^n$, there is $N$ such that for every $m\geq N$, $q^n_m\in j_{m_n,\omega}(D)$. Hence $q^m_m=j_{m_n,m_m}(q^n_m)\in j_{m_n,m_m}(D_0)$ and $j_{m_m,\omega}(q_{m}^m)=q^m_m\in D$. Hence $\vec{q}^*$ diagonalizes $j_{\omega}(\mathcal{F}(\mathbb{Q}))$. Finally, by elementarity, we see that in any generic extension of $\mathbb{P}_{\mathcal{U}}$ we have a generic for $\mathbb{Q}$, and therefore, by Fact \ref{fact: equivalence with generics and weak projections}, there is a weak projection of $\mathbb{P}_{\mathcal{U}}$ to $\ro(\mathbb{Q})$.
\end{proof}\begin{remark}\label{Remark: last remark of section 3}
    The previous theorem, together with Remark \ref{Remark following the RK-construction} shows that the existence of a projection from $\mathbb{P}_{\mathcal{U}}$ to $\ro(\mathbb{Q})$ is equivalent to $\mathcal{F}_q(\mathbb{Q})\leq^{\fin}_{sK}\mathcal{U}$ for every $q\in\mathbb{Q}$. 
\end{remark}
\section{Predistributive Forcings}\label{sec: predistributive Forcings}
In this section we would like to find a combinatorial characterization for the forcing notions of size $\lambda$ which are projections of $\P_{\mathcal{U}}$ where ${\mathcal{U}}$ is some fine measure over $P_\kappa(\lambda)$.

Towards this goal, we introduce the notion of $(\omega,\kappa)$-\emph{predistributivity} and show that it characterizes intermediate models of a broad class of Prikry-type forcings. The precise statement is given in Theorem~\ref{thm: weak dist implies proj} below. We begin by presenting the relevant definitions and discussing their applications.  
\begin{definition}\label{def: predistributive}
   Let $\kappa$ be an uncountable cardinal. A poset $\mathbb{P}$ is $(\omega,\kappa)$-\emph{predistributive} if for each $p\in\mathbb{P}$ and for each $\langle D_\alpha:\alpha<\gamma\rangle\in{}^{<\kappa}\mathcal{F}(\mathbb{P})$, there are $q\leq p$ and $\langle d_n: n<\omega\rangle\in{}^{\omega}\mathcal{F}(\mathbb{P}/q)$ such that for each $\alpha<\gamma$ there is $n<\omega$ with $d_n\s D_\alpha$.
\end{definition}
Observe that $\mathbb{P}$ is $(\omega,\kappa)$-predistributive if and only if $\ro(\mathbb{P})$ is $(\omega,\kappa)$-predistributive. Moreover, if  $\mathbb{P}$ is $(\omega,\kappa)$-predistributive, so is $\mathbb{P}/p$ for each $p\in\mathbb{P}$. It is well known (see Corollary \ref{cor: folk}) that $\kappa$-distributive forcings embed into supercompact Prikry forcing as well as some Prikry-type forcings. 
The next proposition shows that the class of $(\omega,\kappa)$-predistributive forcings catches all these examples under a uniform definition.
We will see later that every forcing in this class is a projection of the strongly compact (supercompact) Prikry forcing. To give a precise formulation of the following proposition, we refer to the class of $\Sigma$-Prikry forcings, whose definition is technically involved. For a complete account, we refer the reader to \cite{poveda2021sigma, poveda2022sigma} and \cite{poveda2023sigma}. To help the reader navigate the section, we focus only on the properties of $\Sigma$-Prikry forcing that are needed for the proof of Proposition \ref{prop: sigma Prikry and kappa-dist}\eqref{item 2: sigma Prikry} below. Accordingly, letting $\Sigma$ be a non-decreasing sequence $\langle\kappa_n: n<\omega\rangle$ of regular uncountable cardinals, a $\Sigma$-Prikry forcing $\mathbb{P}$ is equipped with a surjection $\ell:\mathbb{P}\rightarrow\omega$ such that for all $p,q\in\mathbb{P}$ if $p\leq q$ then $\ell(p)\geq \ell(q)$, and for all $p\in\mathbb{P}$, there is $q\in\mathbb{P}$ with $\ell(q)=\ell(p)+1$. Whenever $m\geq n$, the poset $(\mathbb{P}_m,\leq^\ast)$ is $\kappa_n$-closed, where $p\leq^\ast q\iff p\leq q\, \wedge\, \ell(p)=\ell(q)$ and $\mathbb{P}_m:=\{p\in\mathbb{P}:\ell(p)=m\}$. Similarly, one defines $\mathbb{P}_{\geq m}^q:=\{p\in\mathbb{P}/q: \ell(p)\geq m\}$. Finally, every $\Sigma$-Prikry forcing $\mathbb{P}$ has the \emph{Strong Prikry Property}; i.e., for all $D\in\mathcal{F}(\mathbb{P})$ and all $p\in\mathbb{P}$ there are $q\leq^\ast p$ and $m<\omega$ such that $\mathbb{P}_{\geq m}^q\s D$.

Examples of $\Sigma$-Prikry forcings include Prikry forcing, Supercompact Prikry forcing, Diagonal Supercompact Prikry forcing \cite{gitik-sharon}, AIM forcing \cite{cummings2018ordinal}, and Merimovich Supercompact Extender-Based Prikry forcing \cite{SupercompatRadinExtender}.
\begin{prop}\label{prop: sigma Prikry and kappa-dist}\hfill
    \begin{enumerate}
        \item Every $\kappa$-distributive forcing is $(\omega,\kappa)$-predistributive.
        \item\label{item 2: sigma Prikry} Every $\Sigma$-Prikry forcing is $(\omega,\sup\Sigma)$-predistributive.
    \end{enumerate}
\end{prop}
\begin{proof}
(1) If $\mathbb{P}$ is $\kappa$-distributive, $p\in\mathbb{P}$ and $\langle D_\alpha:\alpha<\gamma\rangle\in{}^{<\kappa}\mathcal{F}(\mathbb{P})$, then $E:=\bigcap_{\alpha<\gamma}D_\alpha\in\mathcal{F}(\mathbb{P})$. Clearly, $\langle E/p: n<\omega\rangle$ is the desired sequence.

(2) Suppose $(\mathbb{P},\ell)$ is $\Sigma$-Prikry. Set $\Sigma:=\langle\kappa_m: m<\omega\rangle$ and $\kappa:=\sup\Sigma$. Pick $p\in\mathbb{P}$, $\langle D_\alpha:\alpha<\gamma\rangle\in{}^{<\kappa}\mathcal{F}(\mathbb{P})$, and $m<\omega$ with $\gamma<\kappa_m$. We may assume (by extending $p$ if necessary) that $\ell(p)\geq m$. For each $\alpha<\gamma$, define $$d_\alpha:=\{q\in\mathbb{P}_{\ell(p)}: \exists k_\alpha<\omega\ \mathbb{P}^q_{\geq k_\alpha}\subseteq D_{\alpha}\}.$$  The Strong Prikry Property ensures that $\{d_\alpha:\alpha<\gamma\}$ is a family of $\leq^\ast$-dense open subsets of $\mathbb{P}_{\ell(p)}$. By $\kappa_m$-distributivity of $(\mathbb{P}_{\ell(p)},\leq^\ast)$, $\bigcap_{\alpha<\gamma}d_\alpha$ is $\leq^\ast$-dense. Thus there is $p^\ast\in\bigcap_{\alpha<\gamma}d_\alpha$ with $p^\ast\leq^\ast p$. Clearly, $\langle \mathbb{P}^{p^\ast}_{\geq n}: n<\omega\rangle$ and $p^\ast$ are as wanted.
\end{proof}
$\Sigma$-Prikry and $\kappa$-distributive forcings have the property that they do not add bounded subsets of $\kappa$. As we will see, ensuring this property for $(\omega,\kappa)$-predistributive forcings is indeed true for some large cardinals (see Corollary \ref{cor: no bounded subsets} below), but does not hold in general (see Question \ref{question: large cardinals}). 
However, in $\zfc$ alone, it turns out that if $\kappa$ is collapsed then its cofinality in the forcing extension has to be $\omega$, provided $\kappa$ is regular in $V$. This emphasizes the dichotomy between $\kappa$-distributivity and $\Sigma$-Prikryness.
\begin{prop}\label{lem: preserving kappa}
    Let $\kappa'$ be a regular cardinal. If $\mathbb{P}$ is a $(\omega,\kappa)$-predistributive forcing which singularizes\footnote{This includes the case that $\mathbb{P}$ collapses $\kappa'$.} $\kappa'$ to have cofinality $\kappa$, then $\mathbb{P}$ forces $\cf(\kappa')=\omega$.
\end{prop}
\begin{proof}

Pick $p \in \mathbb{P}$ forcing the existence of an increasing and cofinal function from a cardinal $\gamma<\kappa'$ to $\kappa'$. 
Let $\dot{f}$ be a $\mathbb{P}$-name such that $$p \Vdash ``\dot{f} \colon \check{\gamma} \to \check{\kappa}' \text{ is increasing and cofinal}".$$ 
For each $\alpha < \gamma$, define $ D_\alpha := \{ r \leq p : r \parallel \dot{f}(\alpha) \}$. 
By $(\omega,\kappa)$-predistributivity, there are $q \leq p$ and $\langle d_n : n < \omega \rangle\in\mathcal{F}(\mathbb{P}/q)$ such that for every $\alpha < \gamma$, there is $n < \omega$ with $d_n \subseteq D_\alpha$. 
For each $n < \omega$, define $E_n := \{ \alpha < \gamma : d_n \subseteq D_\alpha \}$ so that $\bigcup_{n<\omega} E_n = \gamma$. 
Now, let $G\s\mathbb{P}$ be a $V$-generic filter containing $q$, and for each $n < \omega$, pick $r_n \in G \cap d_n$. 
Note that $r_n \in D_\alpha$ for all $\alpha \in E_n$. 
In other words, $r_n$ decides the value of $\dot{f}(\alpha)$ for every $\alpha \in E_n$, and so there is $g_n \colon E_n \to \kappa'$ in $V$ with $r_n \Vdash ``\dot{f} \restriction \check{E}_n = \check{g}_n"$. Since $\kappa'$ is regular in $V$, $\rng(g_n)$ is bounded in $\kappa'$ and let $\gamma_n<\kappa'$ be a bound. Note that in $V[G]$, we have $f := \dot{f}_G=\bigcup_{n<\omega} g_n$, and so $\kappa'= \sup( \bigcup_{n<\omega} \range(g_n))=\sup_{n<\omega}\gamma_n$ as wanted. 
\end{proof}
To find other forcings which are not $(\omega,\kappa)$-predistributive we have the following proposition:
\begin{prop}\label{prop: a weak form of predist}
    Suppose that $\mathbb{P}$ is $(\omega,\kappa)$-predistributive, then for any $\langle D_\alpha\mid \alpha<\gamma\rangle$ dense open for $\omega<\gamma<\kappa$ regular, and any $p$, there is $q\leq p$ and there is $I\in [\gamma]^\gamma$ such that $\bigcap_{\alpha\in I}D_\alpha$ is dense below $q$. 
\end{prop}
\begin{proof}
    Given $\langle D_\alpha\mid \alpha<\gamma\rangle$ dense open for $\omega<\gamma<\kappa$ and some $p$, let $q\leq p$ and $\langle E_n\mid n<\omega\rangle$ be given by the definition of predistributivity. A standard pigeonhole argument gives a single $n<\omega$ and $I\in [\gamma]^\gamma$ such that $E_n\subseteq D_\alpha$ for any $\alpha\in I$.
\end{proof}
\begin{cor}[Compare with \ref{fact: dis sub forcing}]\label{distributive projections}
    Suppose that $\mathbb{P}$ is $(\omega,\kappa)$-predistributive, if $\mathbb{P}$ is $\sigma$-distributive then it is $\kappa$-distributive.
\end{cor}
\begin{example}
$\col(\omega,\omega_2)$ singularizes $\omega_2$ to have cofinality $\omega$, but it is not $(\omega,\omega_2)$-predistributive since for each $\alpha$, letting $D_\alpha=\{p\in\col(\omega,\omega_2)\mid \alpha\in \rng(p)\}$, we have that the intersection of any infinitely-many $D_\alpha$'s   must be empty.  
\end{example}
A similar argument shows that every $(\omega,\kappa)$-predistributive forcing preserves all cardinals below $\kappa$.
\begin{prop}
    If $\P$ is $(\omega,\kappa)$-predistributive and $\gamma<\kappa$ is a cardinal, then $\P$ preserves $\gamma$. In particular, if $\kappa$ is a limit cardinal, then $\P$ preserves $\kappa$. 
\end{prop}
\begin{proof}
    It suffices to prove that all regular cardinals $\gamma<\kappa$ are preserved. Suppose not, and pick $p \in \mathbb{P}$ forcing the existence of a surjection from a cardinal $\delta<\gamma$ onto $\gamma$. Let $\dot{f}$ be a $\mathbb{P}$-name such that $$p \Vdash ``\dot{f} \colon \check{\delta} \to \check{\gamma} \text{ is surjective}".$$ 
For each $\alpha < \gamma$, define $ D_\alpha := \{ r \leq p : \exists \chi <\delta\,  (r \Vdash_{\P} \dot{f}(\check{\chi})=\check{\alpha}) \}$. Since $\gamma$ is regular, for any $I\in [\gamma]^{\gamma}$, the intersection $\bigcap_{\alpha\in I}D_\alpha$ must be empty. This is impossible by Proposition \ref{prop: a weak form of predist}.
\end{proof}
\begin{example} 
    Namba forcing is another classical singularizing forcing.
    In particular, Namba forcing changes the cofinality of $\omega_2$ to $\omega$ while preserving $\omega_1$.
    In contrast with Prikry forcing, Namba forcing is not necessarily $(\omega,\omega_2)$-predistributive. 
    To see this, note that under $\CH$, Namba forcing does not add reals. 
    Since $|\omega^V_2|^{V[G]} = \omega_1$, there is a new subset of $\omega_1$. 
    Therefore, there is a fresh $A\subseteq\omega_1$ (i.e. $A\in V[G]\setminus V$ and $A\cap\alpha\in V$ for every $\alpha<\omega_1$). 
    Letting $D_\alpha$ be the dense set of conditions deciding $A\cap \alpha$, we see there is no $q$ for which the intersection $\omega_1$-many of the $D_\alpha$'s is dense below $q$.
    So by Proposition \ref{prop: a weak form of predist}, Namba forcing is not $(\omega,\omega_2)$-predistributive.
    While we did not know the answer without $\CH$, Andreas Lietz noted that Namba forcing always singularizes $\omega_3$ to have cofinality $\omega_1$ which by Proposition \ref{lem: preserving kappa} is impossible for an $(\omega,\omega_2)$-predistributive forcing.
\end{example}
Below we show that $(\omega,\kappa)$-predistributivity is preserved through projections.
This fact will be key in the rest of this section.
Note that this may fail for $\Sigma$-Prikry forcings.
\begin{prop}\label{thm: closed under proj}
 If $\mathbb{P}$ is $(\omega,\kappa)$-predistributive and $\pi:\mathbb{P}\rightarrow\mathbb{Q}$ is a projection, then $\mathbb{Q}$ is also $(\omega,\kappa)$-predistributive.
\end{prop}
\begin{proof}
    If $q\in\mathbb{Q}$ and $\langle D_\alpha: \alpha<\gamma\rangle\in{}^{<\kappa}\mathcal{F}(\mathbb{Q})$, then $$\langle \pi^{-1}[D_\alpha]: \alpha<\gamma\rangle\in{}^{<\kappa}\mathcal{F}(\mathbb{P}),$$ and $\pi(p)\leq q$ for some $p\in\mathbb{P}$. Since $\mathbb{P}$ is $(\omega,\kappa)$-predistributive, there are $r\leq p$ and $\langle E_n: n<\omega\rangle\in{}^{\omega}\mathcal{F}(\mathbb{P}/r)$ such that for each $\alpha<\gamma$ there is $n<\omega$ with $E_n\s \pi^{-1}[D_\alpha]$. Clearly, $\pi(r)\leq q$, and letting $$d_n:=\{t\in\mathbb{Q}/\pi(r):\exists t'\in\pi``E_n\; (t\leq t')\}$$ for each $n<\omega$, we claim that $\langle d_n: n<\omega\rangle\in {}^{\omega}\mathcal{F}(\mathbb{Q}/\pi(r))$ are suitable. First, if $\alpha<\gamma$ then there is $n<\omega$ with $E_n\s\pi^{-1}[D_\alpha]$, yielding $d_n\s D_\alpha$. It remains to see that $d_n$ are dense below $\pi(r)$. Let $q'\leq \pi(r)$, since $\pi$ is a projection, there is $r'\leq r$ such that $\pi(r')\leq q'$, and since $E_n$ is dense below $r$, there is $r''\leq r'$ such that $r''\in E_n$. Hence $\pi(r'')\in \pi ``E_n$ and $\pi(r'')\leq q'$ establishing that $d_n$ is dense.
\end{proof}
In order to formulate the main result of this section, we need the following definition:
\begin{definition}\label{def trace property}
    Let $M$ be an outer model of $\zfc$ and let $\kappa$ be a cardinal in $M$. 
    \begin{enumerate}
        \item $M$ has the $\kappa$-\emph{trace property} if, for all $f\colon \kappa\rightarrow E$ in $M$ with $E\s V$, there is a sequence $\langle B_n: n<\omega\rangle\in M$ of bounded subsets of $\kappa$ such that $\bigcup_{n<\omega}B_n=\kappa$ and $f\restriction B_n\in V$, for all $n<\omega$.
        \item A forcing notion $\mathbb{P}$ with $\one\Vdash_{\mathbb{P}}``\kappa\text{ is a cardinal}"$ has the $\kappa$-\emph{trace property} if $V[G]$ has the $\kappa$-trace property, for every $\mathbb{P}$-generic filter $G$ over $V$.
    \end{enumerate}
    \begin{remark}
        Note that $M$ has the $\kappa$-trace property if and only if, for all $f\colon \kappa\rightarrow \ord$ in $M$, there is a sequence $\langle B_n: n<\omega\rangle\in M$ of bounded subsets of $\kappa$ such that $\bigcup_{n<\omega}B_n=\kappa$ and $f\restriction B_n\in V$, for all $n<\omega$.
    \end{remark}
    The $\kappa$-trace property has been introduced in \cite{PovedaThei2024baire} as a tool for constructing $\kappa$-perfect sets within the choiceless model $L(V_{\kappa+1})$. On the other hand, $(\omega,\kappa)$-predistributive forcing notions arise, implicitly, in the work of Dimonte-Wu, Trang-Shi, Shi, and Woodin (see respectively, \cite{dimonte2016general, shi20170, woodin2011suitable} and \cite{shi2015axiom}). With attention shifted upward to stronger principles (e.g., $I_0$ and $I_1$), their analysis provides a method for constructing a generic filter over the $\omega^{\text{th}}$ iterate of an $I_0(\kappa)$-embedding within $V$. Thus, both lines of investigation lead to methods for building generic filters within a given model, suggesting a deep connection between these two properties. A first straightforward link is provided by the following fact.
\end{definition}
\begin{fact}[\cite{PovedaThei2024baire}]\label{fact: trace lemma}
    If $\mathbb{P}$ is $\Sigma$-Prikry and forces $\cf(\sup(\Sigma))=\omega$, then $\mathbb{P}$ has the $\sup(\Sigma)$-trace property. 
\end{fact}
In light of Proposition \ref{prop: sigma Prikry and kappa-dist}\eqref{item 2: sigma Prikry}, it is natural to ask whether the assumption of Fact \ref{fact: trace lemma} can be relaxed. We will show later in Lemma \ref{lemma: goodness} that $(\omega,\kappa)$-predistributive suffices to  yield the same conclusion.




 We are now ready to prove the main theorem of this section:
\begin{theorem}\label{thm: weak dist implies proj}
    Suppose that $\mathbb{P}$ has the $\kappa$-trace property and that  $\mathbb{Q}$ is $(\omega,\kappa)$-predistributive. If $\one_\mathbb{P}\forces_\mathbb{P}``\mathcal{F}_q(\mathbb{Q})^V$ is generated by a set of size ${\leq}\kappa$'' for some $q\in\mathbb{Q}$, then for any generic $G\subseteq \mathbb{P}$ there is a countable subset of $\mathbb{Q}$ in $V[G]$ which generates a $V$-generic containing $q$. In particular, there is a weak projection $\pi:\mathbb{P}\rightarrow\ro(\mathbb{Q})$. 
\end{theorem}
\begin{proof}
     Fix a generic $G\s\mathbb{P}$. 
     We show that there is a countable $X\s\mathbb{Q}$ in $V[G]$ such that $\{r\in\mathbb{Q}: \exists x\in X\ (x\leq r)\}$ is $V$-generic and contains $q$. By our assumption $\mathcal{F}_q(\mathbb{Q})^V$ is generated by a set of the form $\langle D_\alpha:\alpha<\kappa\rangle\in V[G]$, and by the $\kappa$-trace property, there is a sequence $\langle B_n: n<\omega\rangle\in V[G]$ with $\bigcup_{n<\omega}B_n=\kappa$ such that for all $n<\omega$, $B_n$ is a bounded subset of $\kappa$, and $\langle D_\alpha: \alpha\in B_n\rangle\in V$. Working in $V[G]$, we construct simultaneously and by induction 
    \begin{enumerate}
    \item a matrix $\{p_{m,n}: m,n<\omega\}$ consisting of conditions in $\mathbb{Q}/q$ such that the $\mathbb{Q}$-upwards closure of the first column, namely $\{r\in\mathbb{Q}: \exists n<\omega\ (p_{n,0}\leq r)\}$, is $V$-generic, and
    \item another matrix $\{ D_m^{n}: m,n<\omega\}$ such that, for all $m, n<\omega$, $D_m^{n}\in\mathcal{F}(\mathbb{Q}/p_{n,0})$ and $p_{n,m+1}\in D_m^{n}$.
    \end{enumerate}
    Throughout the construction we will tacitly appeal to the fact that  $\mathbb{Q}/q$ is $(\omega,\kappa)$-predistributive. Accordingly, pick $p_{0,0}\in \mathbb{Q}/q$ and $\langle D_n^0: n<\omega\rangle\in{}^{\omega}\mathcal{F}(\mathbb{Q}/p_{0,0})$ so that for all $\alpha\in B_0$, there is $n<\omega$ with $D_n^0\s D_\alpha$. By density there is $p_{0,1}\in D_0^0/p_{0,0}$. Pick $p_{1,0}\leq p_{0,1}$ and $\langle D_n^1: n<\omega\rangle\in{}^{\omega}\mathcal{F}(\mathbb{Q}/p_{1,0})$ so that for all $\alpha\in B_1$, there is $n<\omega$ with $D_n^1\s D_\alpha$. Similarly, pick $p_{2,0}\leq p_{1,0}$ and $\langle D_n^2: n<\omega\rangle\in{}^{\omega}\mathcal{F}(\mathbb{Q}/p_{2,0})$ so that for all $\alpha\in B_2$, there is $n<\omega$ with $D_n^2\s D_\alpha$. This completes the first step of the construction. Next, we deal with the conditions $p_{1,1}\geq p_{0,2}\geq p_{0,3}\geq p_{1,2}\geq p_{2,1}\geq p_{3,0}$. As $D_0^1\in\mathcal{F}(\mathbb{Q}/p_{1,0})$ and $p_{2,0}\leq p_{1,0}$ there is $p_{1,1}\in D_0^1/p_{2,0}$. Similarly, $D_1^0\in\mathcal{F}(\mathbb{Q}/p_{0,0})$ and $p_{1,1}\leq p_{0,0}$ ensure the existence of $p_{0,2}\in D_1^0/ p_{1,1}$. $D_2^0\in\mathcal{F}(\mathbb{Q}/p_{0,0})$ yields $p_{0,3}\in D^0_2/p_{0,2}$, $D_1^1\in\mathcal{F}(\mathbb{Q}/p_{1,0})$ yields $p_{1,2}\in D^1_1/p_{0,3}$, and $D_0^2\in\mathcal{F}(\mathbb{Q}/p_{2,0})$ yields $p_{2,1}\in D_0^2/p_{1,2}$. Finally, pick $p_{3,0}\leq p_{2,1}$ and $\langle D_n^3: n<\omega\rangle\in{}^{\omega}\mathcal{F}(\mathbb{Q}/p_{3,0})$ so that for all $\alpha\in B_3$, there is $n<\omega$ with $D_n^3\s D_\alpha$.
    
    We continue the construction in the very same fashion. In the first column we have witnesses of the fact that $\mathbb{Q}$ is $(\omega,\kappa)$-predistributive so that we generate the $p_{n,0}$'s and the $\langle D_m^n: m<\omega\rangle$'s, while in the other columns we have witnesses of the densities of the various $D_m^n\in\mathcal{F}(\mathbb{Q}/p_{n,0})$. Specifically, if $n>3$ is even, pick $p_{n,0}\leq p_{n-1,0}$ and $\langle D_m^n: m<\omega\rangle\in{}^{\omega}\mathcal{F}(\mathbb{Q}/p_{n,0})$ so that for all $\alpha\in B_n$, there is $m<\omega$ with $D_m^n\s D_\alpha$. Once the matrix is defined up to $p_{n,0}$, for $n$ even, we construct the (finite) snake $$p_{n-1,1}\geq p_{n-2,2}\geq\cdots \geq p_{0,n}\geq p_{0,n+1}\geq p_{1,n}\geq p_{2,n-1}\geq \cdots\geq p_{n,1}$$ so that $p_{m,i}\in D^{m}_{i-1}$, whenever $$p_{m,i}\in\{p_{n-1,1}, p_{n-2,2},\cdots, p_{0,n}, p_{0,n+1},p_{1,n}, p_{2,n-1}, \cdots, p_{n,1}\}.$$ On the other hand, if $n>3$ is odd, pick $p_{n,0}\leq p_{n-1,1}$ and $\langle D_m^n: n<\omega\rangle\in{}^{\omega}\mathcal{F}(\mathbb{Q}/p_{n,0})$ so that for all $\alpha\in B_n$, there is $m<\omega$ with $D_m^n\s D_\alpha$. Finally, since $n+1$ is even, define $p_{n+1,0}\leq p_{n,0}$ as explained in the case above.
    
    Now let $h$ be the $\mathbb{Q}$-upwards closure of the first column $\langle p_{n,0}: n<\omega\rangle$, i.e. $h:=\{r\in\mathbb{Q}: \exists n<\omega\ (p_{n,0}\leq r)\}$. Since $\langle p_{n,0}: n<\omega\rangle$ is $\leq$-decreasing, it is clear that $h$ is a filter. Moreover, $p_{0,0}\leq q$ and so $q\in h$ . We verify that $h$ is $V$-generic. Letting $D\in\mathcal{F}_q(\mathbb{Q})$, there are $\alpha<\kappa$ and $n<\omega$ such that $D_\alpha\subseteq D$ and $\alpha\in B_n$. By construction, there is $m<\omega$ such that $p_{n,m+1}\in D_m^n\s D_\alpha\s D$. But note that $p_{n,m+1}\in h$, since there is a big enough $i<\omega$ with $p_{i,0}\leq p_{n,m+1}$. Thus, $p_{n,m+1}$ witnesses that $h\cap D\neq\emptyset$. 
\end{proof}
Theorem \ref{thm: weak dist implies proj} leads to a series of quotable corollaries.
\begin{cor}\label{cor of the main thm}
    Suppose that $\mathbb{P}$ has the $\kappa$-trace property and that  $\mathbb{Q}$ is $(\omega,\kappa)$-predistributive. If $\one_\mathbb{P}\forces_\mathbb{P}``\mathcal{F}(\mathbb{Q})^V$ is generated by a set of size ${\leq}\kappa$'', then for any generic $G\subseteq \mathbb{P}$ and any $q\in\mathbb{Q}$ there is a countable subset of $\mathbb{Q}$ in $V[G]$ which generates a $V$-generic containing $q$.
\end{cor}
\begin{cor}\label{cor for supercompact}
    Suppose that $\mathbb{Q}$ is $(\omega,\kappa)$-predistributive and $\lambda$-pregenerated. Then $\ro(\mathbb{Q})$ is a weak projection of the supercompact Prikry $\mathbb{P}_{\mathcal{U}}$, where $\mathcal{U}$ is a normal measure over $P_\kappa(\lambda)$.
\end{cor}
If we moreover assume that $\mathbb{Q}$ has size $\leq\lambda$, then we actually get a projection from $\mathbb{P}_{\mathcal{U}}$ into $\ro(\mathbb{Q})$.
\begin{cor}\label{cor: main thm applied to sc Prikry}
Let $\mathcal{U}$ be a strongly compact (supercompact) measure over $P_\kappa(\lambda)$ and let $\Q$ be an $(\omega,\kappa)$-predistributive, $\lambda$-pregenerated forcing of cardinality ${\leq}\lambda$, then $\ro(\mathbb{Q})$ is a projection of $\mathbb{P}_{\mathcal{U}}$.
\end{cor}
\begin{proof}
 It is clear that $\mathbb{P}_{\mathcal{U}}$ is not $\lambda$-cc. On the other hand, Corollary \ref{cor of the main thm} leads to $\one\forces_{\mathbb{P}_{\mathcal{U}}}``\forall q\in\mathbb{Q}\exists h\, (h$ is $\mathbb{Q}$-generic and $q\in h)"$. So by Proposition \ref{Proposition: turnining a weak projection to a projection}, the conclusion follows.
\end{proof}
Under full strong compactness (supercompactness), we may characterize the class of $(\omega,\kappa)$-predistributive forcings. More importantly, we characterize the intermediate models. 
\begin{cor}\label{cor: no bounded subsets}
    Let $\kappa$ be strongly compact (supercompact). A forcing $\mathbb{Q}$ is $(\omega,\kappa)$-predistributive if and only if for some $\lambda\geq\kappa$ and some strongly compact (supercompact) measure $\mathcal{U}$ over $P_\kappa(\lambda)$, there is a projection from the strongly compact (supercompact) Prikry forcing with
respect to $\mathcal{U}$ into $\ro(\mathbb{Q})$. In particular, $(\omega,\kappa)$-predistributive forcings do not add bounded subsets of $\kappa$.
\end{cor}
\begin{proof}
    Let $\mathbb{Q}$ be $(\omega,\kappa)$-predistributive, and let $\mathcal{U}$ be a strongly compact (supercompact) measure over $P_\kappa(\lambda)$, where $\lambda>\kappa$ is any cardinal such that $\mathbb{Q}$ is $\lambda$-pregenerated and $\lambda\geq |\mathbb{Q}|$. Letting $\mathbb{P}_{\mathcal{U}}$ be the strongly compact (supercompact) Prikry forcing with
respect to $\mathcal{U}$, it is easy to see that Corollary \ref{cor: main thm applied to sc Prikry} yields the desired conclusion. The converse implication is the content of Proposition \ref{thm: closed under proj}.
\end{proof}
Observe that there are cardinals for which the last implication of Corollary \ref{cor: no bounded subsets} may fail. Namely, there are cardinals $\kappa$ and $(\omega,\kappa)$-predistributive forcings adding bounded subsets of $\kappa$. A trivial example is $\kappa=\omega_1$, as every forcing is $(\omega,\omega_1)$-predistributive. 
\begin{question}\label{question: large cardinals}
    For which cardinals $\kappa>\omega$ is the class of $(\omega,\kappa)$-predistributive forcings  a subclass of the $(\kappa,2)$-distributive forcings? 
\end{question}
Below a measurable cardinal our understanding of this class is much narrower and we currently cannot even separate it from, say, the distributive forcings. We leave that as an interesting avenue for future research. The following result gives a consistency answer (under Martin's Axiom-- see e.g. \cite[Ch.II~\S 2]{Kunen}) to the above question for cardinals which are consistently below the continuum, and is due to Andreas Lietz:
\begin{prop}[Lietz]
Under $\mathrm{MA}_{<\kappa}$, Cohen forcing $\Add(\omega,1)$ is $(\omega,\kappa)$-predistributive.    
\end{prop}
\begin{proof}
    Suppose that $\omega<\gamma<\kappa$ is a regular cardinal and we are given a $\gamma$-sequence of dense open sets $\langle D_\alpha\mid \alpha<\gamma\rangle$ in $\Add(\omega,1)$, and a condition $p$, we would like to find $q\leq p$ and $\omega$-many dense sets below $q$ $\langle E_n\mid n<\omega\rangle$ such that for every $\alpha<\gamma$ there is $n$ such that $E_n\subseteq D_\alpha$. Let us present a c.c.c forcing $\mathbb{P}$  which is designated to add the $E_n$'s. Enumerate $\Add(\omega,1)/p$ by $\{p_n\mid n<\omega\}$. A condition is a pair $p=(c,\rho)$ such that:
    \begin{enumerate}
        \item $c:F\to\omega$, where $F\in [\gamma]^{<\omega}$ ($c(\alpha)=n$ is a promise that $E_n\subseteq D_\alpha$).
        \item $\rho:A\to \Add(\omega,1)$, where $A\in [\omega\times\omega]^{<\omega}$ (so that $E_n$ will consist of the $\rho(n,i)$, ranging over $i$).
    \end{enumerate}
    We denote the objects above by $c=c^p,\rho=\rho^p,F=F^p,A=A^p$.
    To be a condition we require that $p$ satisfy the following:
    \begin{enumerate}
        \item [(I)]for each $(n,i)\in A^p$, $\rho^p(n,i)\leq p_i$.
        \item [(II)] for each $\alpha\in F^p$, if $c^p(\alpha)=n$, then $\rho^p(n,i)\in D_\alpha$.
    \end{enumerate} 
    The order is end-extension on both coordinates.
To see that $\mathbb{P}$ is c.c.c., let $\{p_i\mid i<\omega_1\}\subseteq \mathbb{P}$, since there are only countably many choices for $\rho^p$, we may assume that for every $i<\omega_1$, $\rho^{p_i}=\rho$. Apply the $\Delta$-system lemma to $F^p$ to find a root $r$, then for every $i\neq j$ such that $F^{p_i}\cap F^{p_j}=r$ and $c^{p_i}\restriction r=c^{p_j}\restriction r$ we define $c^p=c^{p_i}\cup c^{p_j}$ and $\rho^p=\rho$. It is easy to see that conditions $(I),(II)$ are satisfied and so the forcing is c.c.c.

The following $\gamma$-many dense sets are what we need:
\begin{itemize}
    \item $X_\alpha=\{p\mid \alpha\in F^p\}$ for every $\alpha<\gamma$.
    \item $Y_{(n,i)}=\{p\mid (n,i)\in A^p\}$ for every $(n,i)\in \omega\times\omega$.
\end{itemize}
By $\mathrm{MA}_{<\kappa}$, there is a filter $g\subseteq \mathbb{P}$ meeting those dense sets, we can define $c^g=\bigcup_{p\in g}c^p$, and $\rho^g=\bigcup_{p\in g}\rho^p$. Let $E^g_n=\{\rho^g(n,i)\mid i<\omega\}$, then by condition $(I
)$, each $E^g_n$ is dense. By condition $(II)$, if $c^g(\alpha)=n$, then $E^g_n\subseteq D_\alpha$.
\end{proof}
Recall that the class of forcings not adding bounded subsets of $\kappa$ coincides with the class of $(\kappa,2)$-distributive forcings (see \cite{JECH198411} for a full account). A useful characterization of $(\kappa,2)$-distributivity under strong compactness has been provided by Ben Neria, Goldberg and Kaplan.
\begin{theorem}[\cite{BenNeria_Goldberg_Kaplan}]\label{thm: BGK}
    If $\kappa$ is strongly compact, then a partial order is $(\kappa,2)$-distributive if and
only if it is forcing-equivalent to a $\kappa$-complete Prikry-type poset.
\end{theorem}
Therefore, in light of Corollary \ref{cor: no bounded subsets} and Theorem \ref{thm: BGK} the following question is natural.
\begin{question}
    Suppose that $\kappa$ is strongly compact (supercompact) and that $\mathbb{Q}$ is $(\kappa,2)$-distributive. Is there a weak projection from a strongly compact (supercompact) Prikry forcing into $\ro(\mathbb{Q})$?
\end{question}

\begin{prop}
    Assume that $\kappa$ is $\lambda^+$-supercompact and $2^\lambda=\lambda^+$. Then for every $(\omega,\kappa)$-predistributive forcing $\mathbb{Q}$ of cardinality $\lambda$, there is a $\lambda$-strongly compact measure $\mathcal{U}$ on $P_\kappa(\lambda)$ such that $\ro(\mathbb{Q})$ is a projection of $\mathbb{P}_{\mathcal{U}}$.
\end{prop}
\begin{proof}
    We may assume that $\mathbb{Q}=\lambda$. Suppose that $\kappa$ is $\lambda^+$-supercompact, and  $2^\lambda=\lambda^+$. Let $\mathcal{W}$ be a $\lambda^+$-supercompact measure.  In $M_{\mathcal{W}}$, apply the fact that $j_{\mathcal{W}}(\mathbb{Q})$ is $(\omega,j_U(\kappa))$-predistributive to find $(d_n: n<\omega)$, such that for every $D\in\mathcal{F}_q(\mathbb{Q})$ there is $n<\omega$ such that $d_n\subseteq j_{\mathcal{W}}(D)$. Now we can recursively construct a sequence $p_0\geq p_1...$ such that $p_n\in d_n$ and for every $D$ there is a tail of $n$'s such that $p_n\in j_{\mathcal{W}}(D)$. Similar to the proof of Corollary \ref{cor: distibutive consistency}, we can derive a fine $\kappa$-complete measure $\mathcal{U}$ on $P_\kappa(\lambda)$, such that $p_n\in \text{rng}(k)$ for every $n$, where $k:M_{\mathcal{U}}\to M_{\mathcal{W}}$ is the factor map. Then Prikry forcing with $\mathcal{U}$ will project on $\ro(\mathbb{Q})$ by Theorem \ref{Thm: more general equivalence} (and Remark \ref{Remark: last remark of section 3}).
\end{proof}
\begin{remark}\label{remak: large cardinal assumption}
 A sufficient large cardinal assumption for this theorem is a lightface $\lambda$-$\Pi^1_1$-subcompact cardinal\footnote{We say that $\kappa$ is a lightface $\lambda$-$\Pi^1_1$-subcompact cardinal if for any $\Pi^1_1$-sentence $\Psi$, if $\langle H(\lambda^{+}),\in {\rangle}\models \Psi$, there are $\rho<\bar{\lambda}<\kappa$ such that $\langle H(\bar{\lambda}),\in{\rangle}\models \Psi$ and an elementary embedding $j:\langle H(\bar{\lambda}),\in {\rangle}\to \langle H(\kappa^{+}),\in {\rangle}$ such that $j(\rho)=\kappa$.} (see \cite[Def. 36]{Benhamou_Hayut_Gitik}). Indeed, the proof is completely analogous to that of \cite[Prop.37]{Benhamou_Hayut_Gitik}.
\end{remark}

We conclude these corollaries with the full level-by-level characterization of the subforcings of the strongly compact Prikry forcing of cardinality $\lambda$.
\begin{cor}\label{LevelByLevel Cor}
    Suppose that $\kappa$ is $2^\lambda$-strongly compact. Among the forcing $\mathbb{Q}$ of cardinality $\lambda$, the projections of a $\lambda$-strongly compact Prikry forcing are exactly the $(\omega,\kappa)$-predistributive forcings.
\end{cor}
In the next section, we will provide several other characterizations of the $(\omega,\kappa)$-predistributive class.
The idea of constructing generic filters in \ref{thm: weak dist implies proj}, can be also used in the ground model to obtain partially generic filters:
\begin{prop}
     If $\mathbb{P}$ is $(\omega,\kappa)$-predistributive then $\mathrm{MA}_{<\kappa}(\mathbb{P})$ holds, namely, for any fewer than $\kappa$-may dense sets in $\mathbb{P}$ there is a filter $G$ on $\mathbb{P}$ which meets all the dense sets.
 \end{prop}
 \begin{proof}
     We will show a stronger statement, that there is a countably generated filter which meets a given list of dense sets. Let $\langle D_\alpha\mid \alpha<\gamma\rangle$ for some $\gamma<\kappa$ be a sequence of dense sets in $\mathbb{Q}$ Applying  $(\omega,\kappa)$-predistributivity, we can find a witnessing  $p$ and sets $E_n$ dense below $p$. Constructing any decreasing sequence $p_0\geq p_1\geq...$ such that $p_n\in E_n$, we conclude that the filter generated by the $p_n$'s intersect every $D_\alpha$. 
 \end{proof}
 The proof above, together with the argument in \ref{thm: weak dist implies proj}, also shows that if $\kappa$ is singular of countable cofinality, then $(\omega,\kappa)$-predistributive forcings satisfy $MA_{\kappa^+}$. As a result we see that if $\mathbb{P}$ is a non-atomic forcing which is $\lambda$-pregenerated, then for every $\lambda<\kappa$, $\mathbb{P}$ is not $(\omega,\kappa)$-predistributive. 

For the rest of this section, we further examine $(\omega,\kappa)$-predistributive forcings and models with the $\kappa$-trace property, focusing in particular on the relationships between them. 
We start with the following lemma:
\begin{lemma}[Trace Lemma]\label{lemma: goodness}
     If $\mathbb{P}$ is a $(\omega,\kappa)$-predistributive forcing notion preserving $\kappa$, then the following are equivalent: \begin{enumerate}
         \item $\one_\P\Vdash_{\mathbb{P}}``\cf(\kappa)=\omega$''.
         \item $\mathbb{P}$ has the $\kappa$-trace property.
     \end{enumerate}
\end{lemma}
\begin{proof}
Let $G\s\mathbb{P}$ be $V$-generic. By the $\kappa$-trace property applied to $\id:\kappa\rightarrow\kappa$, there is a sequence $\langle B_n: n<\omega\rangle\in V[G]$ of bounded subsets of $\kappa$ such that $\bigcup_{n<\omega}B_n=\kappa$. Thus $V[G]\models\cf(\kappa)=\omega$, as witnessed by $\langle\sup(B_n): n<\omega\rangle$.  

Conversely, let $f\colon \kappa\rightarrow A$ be a function in $V[G]$ with $A\s V$. Work in $V[G]$ and fix $\langle \nu_n: n<\omega\rangle$ an increasing cofinal sequence in $\kappa$. In the proof we define auxiliary sequences $\langle r_{n,i}: i,n<\omega\rangle$ and $\langle C_{n,i}: i,n<\omega\rangle$ such that the following properties hold for each $n<\omega:$
\begin{enumerate}
    \item[$(\alpha)$] $\langle r_{n,i}: i<\omega\rangle\in{}^\omega G$ is a $\leq$-decreasing sequence.
    \item[$(\beta)$] $\langle C_{n,i}: i<\omega\rangle$ is a $\subsetneq$-increasing sequence of subsets of $[\nu_{n-1},\nu_n)$ such that $\bigcup_{i<\omega}C_{n,i}=[\nu_{n-1},\nu_n).$ Here, by convention, $\nu_{-1}:=0.$
    \item[$(\gamma)$] $r_{n,i}\Vdash_{\mathbb{P}}``\dot{f}\restriction C_{n,i}\in\check{V}$'', where $\dot{f}$ is a $\mathbb{P}$-name such that $\dot{f}_G=f.$ 
\end{enumerate} Fix $n<\omega$ such that $f\restriction [\nu_{n-1},\nu_n)\notin V$ -- otherwise one simply sets $B_n:=[\nu_{n-1},\nu_n)$ -- and let $p\in G$ and $\tau$ be a $\mathbb{P}$-name such that $p\Vdash_{\mathbb{P}}``\tau=\dot{f}\restriction [\nu_{n-1},\nu_n)$''. The construction of the above sequences is accomplished by induction on $i<\omega.$

\smallskip

\underline{Case $i=0$:} For each $\xi\in [\nu_{n-1},\nu_n)$ let $D_\xi:=\{r\in\mathbb{P}: r\parallel \tau(\xi)\}$, and for each $r\in\mathbb{P}$ define $C_r:=\{\xi\in [\nu_{n-1},\nu_n): r\parallel\tau(\xi)\}$. 

\begin{claim}\label{claim: D_0 is dense below p}
  Let $E_0$ be the set of all conditions $r\leq p$ such that: \begin{enumerate}
        \item\label{def of D_0, 1} There is $\langle d_j: j<\omega\rangle\in{}^{\omega}\mathcal{F}(\mathbb{P}/r)$ such that for each $\xi\in [\nu_{n-1},\nu_n)$ there is $j<\omega$ with $d_j\s D_\xi$; 
        \item\label{def of D_0, 2} $C_r$ is non-empty and $r\Vdash_{\mathbb{P}}\tau\restriction\check{C}_{r}\in \check{V}$.
    \end{enumerate}
    Then $E_0$ is dense below $p\in G.$
\end{claim}
\begin{proof}[Proof of claim]
    Fix $q\leq p$. Since $\mathbb{P}$ is $(\omega,\kappa)$-predistributive, we may find a condition $q'\leq q$ witnessing Clause \eqref{def of D_0, 1}. Now we $\leq$-extend $q'$ to $r$ to make sure that $r\Vdash_{\mathbb{P}} \tau(\check{\nu}_{n-1})=\check{e}_{\nu_{n-1}}$ for some $e_{\nu_{n-1}}\in V$, which yields $C_r\neq \emptyset$. Set $h:=\{\langle \xi, e_\xi\rangle\in [\nu_{n-1},\nu_n)\times V : \xi\in C_r\}$, where $e_\xi$ is the unique $e\in V$ with $r\Vdash_{\mathbb{P}}\tau(\check{\xi})=\check{e}$. We conclude that $r\Vdash_{\mathbb{P}} \tau\restriction \check{C}_r=\check{h}\in \check{V}$, and so also Clause \eqref{def of D_0, 2} is met.
\end{proof}
 Stipulate $r_{n,0}\in G\cap E_0$, $C_{n,0}:=C_{r_{n,0}}$, and let $\langle d_j: j<\omega\rangle$ be the sequence witnessing Clause \eqref{def of D_0, 1}.

\smallskip

\underline{Induction step:} Suppose that $\langle r_{n,i}: i\leq j\rangle$ and $\langle C_{n,i}: i\leq j\rangle$ have been constructed complying with Clauses $(\alpha)$, $(\beta)$ and $(\gamma)$ above, for $j<\omega$. Note that $C_{n,j}$ is not $[\nu_{n-1}, \nu_{n})$ for otherwise $r_{n,j}\in G$ would have decided  the value of $f\restriction [\nu_{n-1},\nu_n)$, contrary to our departing assumption. As a result the set
$$E_{j+1}:=\{r\leq r_{n,j}: r\in d_j\;\wedge\; r\Vdash_{\mathbb{P}}\text{$``\tau\restriction \check{C}_r\in \check{V}$''}\;\wedge\; C_{n,j}\subsetneq C_r\}$$
is dense below $r_{n,j}\in G$. So again we can pick $r_{n,j+1}\in G\cap E_{j+1}$ and set $C_{n,j+1}:=C_{r_{n,j+1}}.$

\medskip

This finishes the construction of $\langle r_{n,i}, C_{n,i}: i<\omega\rangle$ witnessing $(\alpha)$, $(\beta)$ and $(\gamma)$.

\begin{claim}
    $\bigcup_{i<\omega}C_{n,i}=[\nu_{n-1},\nu_n)$ and $f\restriction C_{n,i}\in V$, for all $i<\omega$.
\end{claim}
\begin{proof}[Proof of claim]
    By construction,  $r_{n,i}\Vdash_{\mathbb{P}} \dot{f}\restriction \check{C}_{n,i}=\tau\restriction \check{C}_{n,i}\in \check{V}$ hence (as $r_{n,i}\in G$) $f\restriction C_{n,i}\in V.$ For the other claim, suppose $\chi\in [\nu_{n-1},\nu_n)$, and pick $j<\omega$ with $d_j\s D_\chi$. By construction, $r_{n,j+1}\in E_{j+1}$, and so $r_{n,j+1}\in d_j$. In particular, $r_{n,j+1}\in D_\chi=\{r\in\mathbb{P}: r\parallel\tau(\chi)\}$. Recalling that $C_{r_{n,j+1}}=\{\xi\in[\nu_{n-1},\nu_n): r_{n,j+1}\parallel\tau(\xi)\}$, it follows that $\chi\in C_{r_{n,j+1}}=C_{n,j+1}$.
\end{proof}
The above argument was carried out for a fixed $n<\omega$ so 
we obtain a sequence 
$\langle C_{n,i}: i,n<\omega\rangle$ such that $\bigcup_{i<\omega}C_{n,i}:=[\nu_{n-1},\nu_n)$. Pick a  bijection $F\colon \omega\rightarrow\omega\times\omega$ (e.g., the inverse of G\"odel pairing
function) and, for all $m<\omega$, stipulate $$B_m:=\begin{cases}
   [\nu_{m-1},\nu_m), & \text{if $f\restriction [\nu_{m-1},\nu_m)\in V$};\\
   C_{F(m)}, & \text{otherwise.}
\end{cases}
$$
Clearly $\langle B_m: m<\omega\rangle$ is as desired; namely $\bigcup_{m<\omega} B_m=\kappa$, each $B_m$ is bounded, and $f\restriction B_m\in V$ for each $m<\omega$.  
\end{proof}
\begin{cor}
Let $\mathcal{U}$ be a strongly compact (supercompact) measure over $P_\kappa(\lambda)$ and let $\mathbb{Q}$ be a forcing notion. If there is a projection from $\mathbb{P}_{\mathcal{U}}$ into $\mathbb{Q}$ and $\one_\mathbb{Q}\Vdash_{\mathbb{Q}}``\cf(\kappa)=\omega$'', then $\mathbb{Q}$ has the $\kappa$-trace property.
\end{cor}
\begin{proof}
    By Proposition \ref{thm: closed under proj}, $\mathbb{Q}$ is $(\omega,\kappa)$-predistributive and clearly preserves $\kappa$. Thus, Lemma \ref{lemma: goodness} applies, yielding the $\kappa$-trace property for $\mathbb{Q}$.
\end{proof} 

The $\kappa$-trace property admits a combinatorial characterization arising from a weakening of $(\omega,\kappa)$-predistributivity. Such a weakening comes up naturally if we consider the following characterization, noticed by Goldberg, of $(\omega,\kappa)$-predistributive forcings:
\begin{prop}[Goldberg]\label{prop: equiv}
Let $\mathbb{Q}$ be a separative forcing notion. The following are equivalent:
\begin{enumerate}
    \item $\mathbb{Q}$ is $(\omega,\kappa)$-predistributive.
    \item For every $V$-generic $G\subseteq \mathbb{Q}$, any $\gamma<\kappa$ and any function $t:\gamma\to \ord$, there is  $\langle B_n: n<\omega\rangle\in V$ such that $\gamma=\bigcup_{n<\omega} B_n$ and for every $n<\omega$, $t\restriction B_n\in V$.
\end{enumerate}
\end{prop}
\begin{proof}
First we verify $(1)\to (2)$. Let $\dot{t}$ be a $\mathbb{Q}$-name for $t$ and let $p\in G$ force that $p\Vdash_{\mathbb{Q}}\dot{t}:\check{\gamma}\to \ord$. For each $\alpha$, let $D_\alpha$ be the dense set below $p$ of conditions deciding $\dot{t}(\alpha)$. By $(\omega,\kappa)$-predistributivity there is $q\leq p$ and a sequence $\langle E_n: n<\omega\rangle$ dense below $q$ such that for every $\alpha$ there is $n$ such that $E_n\subseteq D_\alpha$. Let $B_n=\{\alpha<\gamma: E_n\subseteq D_\alpha\}$. Then $\bigcup_n B_n=\gamma$. By density, we may find such a $q$ in the generic $G$. Take any $q_n\in E_n\cap G$. Then for each $\alpha\in B_n$, $q_n\in D_\alpha$ and therefore $q_n$ decides $\dot{t}(\alpha)=t(\alpha)$. Hence $t\restriction B_n\in V$ for every $n<\omega$.
    To see $(2)\to (1)$, let $\langle D_\alpha: \alpha<\gamma\rangle$ be a collection of dense open sets and $p$ be a condition. Let $G$ be generic with $p\in G$ and let $t(\alpha)$ be a choice for a condition in $G\cap D_\alpha$ (this can easily be coded in $V$ as an ordinal function). Then by $(2)$, there is $\langle B_n\mid n<\omega\rangle$  such that $\gamma=\bigcup_nB_n$ and $t\restriction B_n\in V$ for every $n$.
    Let $\dot{t}$ be a name such that $(\dot{t})_G=t$ and let $q\leq p$, $q\in G$ force that for every $n<\omega$, $\dot{t}\restriction B_n\in V$, and for every $\alpha$, $\dot{t}(\alpha)\in D_\alpha$. Let $E_n$ be the dense open set below $q$ of conditions which decide $\dot{t}\restriction B_n$. Let $\alpha<\gamma$. Then there is $n<\omega$ such that $\alpha \in B_n$. We claim that $E_n\subseteq D_\alpha$, indeed, for each $q'\in E_n$, there is $x\in D_\alpha$ such that $q'\Vdash \dot{t}(\alpha)=x\in \dot{G}\cap D_\alpha$. Since $q'\Vdash x\in \dot{G}$, $q'\leq x$ (we assume separativity), and since $D_\alpha$ is open $q'\in D_\alpha$ as wanted.
 \end{proof}

 By Proposition \ref{prop: equiv} the connection with the trace property is more evident and suggests the following natural weakening:
 \begin{definition}\label{def: weakly pre gen}
     Let $\kappa$ be an uncountable cardinal. A forcing $\mathbb{P}$ is \emph{weakly} $(\omega,\kappa)$-\emph{predistributive} if for any $V$-generic $G\s\mathbb{P}$ and any function $f\in{}^{<\kappa}V$ in $V[G]$ there is  $\langle B_n: n<\omega\rangle\in V[G]$ such that $\dom(f)=\bigcup_n B_n$ with $f\restriction B_n\in V$ for each $n<\omega$.
 \end{definition}
 As expected, there is a combinatorial definition of weakly $(\omega,\kappa)$-predistributive forcings, along the line of Proposition \ref{prop: equiv}.
\begin{lemma}
     Let $\mathbb{P}$ be a forcing poset. Then the following are equivalent: \begin{enumerate}
         \item $\mathbb{P}$ is weakly $(\omega,\kappa)$-predistributive.
         \item  For every $\langle D_\alpha: \alpha<\gamma\rangle\in{}^{<\kappa}\mathcal{F}(\mathbb{P})$, there are countably many maximal antichains $\langle A_n: n<\omega\rangle$, such that for every $\alpha<\gamma$, $D_\alpha\cap(\bigcup_{n<\omega}A_n)$ is predense.
     \end{enumerate}
 \end{lemma}
 \begin{proof}
    First, we verify $(2)\rightarrow (1)$. Let $f\in{}^{<\kappa}V$ be a function in $V[G]$. We can find a name $\dot{f}$ such that $\one_{\P}\Vdash\dot{f}:\check{\gamma}\to V$ for some $\gamma<\kappa$.  For each $\alpha<\gamma$ let $D_\alpha=\{q: q || \dot{f}(\alpha)\}$. By the assumption there is a sequence $\langle A_n\mid n<\omega\rangle$ of antichains witnessing $(2)$. For  each $n$, let $q_n\in G\cap A_n$, and consider $B_n=\{\alpha: q_n\in D_\alpha\}$. Then each $B_n\in V$, but since the sequence $\langle q_n: n<\omega\rangle\in V[G]$ we can only guarantee that $\langle B_n: n<\omega\rangle\in V[G]$. To see that $\langle B_n: n<\omega\rangle$ is a partition of $\gamma$, pick $\alpha<\kappa$. By our assumption, $D_\alpha\cap (\bigcup_{n<\omega}A_n)$ is predense, so there is $q'\in G\cap (\bigcup_{n<\omega}A_n)\cap D_\alpha$. Let $n$ be such that $q'\in A_n$. Since $A_n$ is an antichain, $q'=q_n$, which by definition means that $\alpha\in B_n$. 
    
   To see $(1) \rightarrow (2)$, let $\langle D_\alpha: \alpha<\gamma\rangle$ be a sequence of dense open sets, with $\gamma<\kappa$, and let $G$ be $V$-generic. Set $f:\gamma\to G\subseteq V$ to be a choice function sending $\alpha$ to $p_\alpha\in D_\alpha\cap G$. By our assumption, there is a partition $\langle B_n: n<\omega\rangle\in V[G]$ of $\gamma$, such that for every $n<\omega$, $f\restriction B_n\in V$. This means that there is $q_n\in G$ such that $q_n\Vdash \rng(f\restriction B_n)\subseteq G$ which in turn implies (by separativity) that $q_n\leq \rng(f\restriction B_n)$. Let $\dot{q_n}$ be a name for $q_n$ such that $\one_{\P}\Vdash``$ for every $\alpha$ there is $n$ such that $\dot{q}_n\in D_\alpha"$. Let $$E_n=\{q: q\Vdash \dot{q}_n\geq q\}.$$ Then by Zorn's lemma, $E_n$ contains a maximal antichain $A_n$. To conclude, we claim that for each $\alpha$, $D_\alpha\cap(\bigcup_{n<\omega}A_n)$ is predense. Indeed, let $r$ be any condition and let $H$ be generic with $r\in H$. Then there is $n$ such that $(\dot{q}_n)_{H}\in D_\alpha$ and therefore there is $r'\leq(\dot{q}_n)_H,r$ which is in $E_n$.     
 \end{proof}
 We can now state and prove a combinatorial characterization of the trace property. 
 \begin{prop}
     Let $\mathbb{P}$ be a forcing poset preserving $\kappa$. Then the following are equivalent: \begin{enumerate}
         \item $\one_\P\Vdash_{\mathbb{P}}``\cf(\kappa)=\omega$'' and $\mathbb{P}$ is weakly $(\omega,\kappa^+)$-predistributive.
         \item $\mathbb{P}$ has the $\kappa$-trace property.
     \end{enumerate}
 \end{prop}
 \begin{proof}
Suppose $\mathbb{P}$ has the $\kappa$-trace property. In the proof of Lemma \ref{lemma: goodness}, we have already observed that $\one\Vdash_{\mathbb{P}}``\cf(\kappa)=\omega$''. Moreover, it is clear by Definition \ref{def: weakly pre gen} that $\mathbb{P}$ is weakly $(\omega,\kappa^+)$-predistributive. 

Conversely, suppose $\mathbb{P}$ is weakly $(\omega,\kappa^+)$-predistributive let $G\s \mathbb{P}$ be $V$-generic, and let $f:\kappa\to E$ be a function in $V[G]$ with $E\s V$. Since $\kappa$ has countable cofinality in $V[G]$ we may pick a countable cofinal sequence $\langle\nu_n: n<\omega\rangle\in V[G]$ with limit $\kappa$. Since $\mathbb{P}$ is weakly $(\omega,\kappa^+)$-predistributive, there is $\langle B_m: m<\omega\rangle\in V[G]$ of $\kappa$ with $f\restriction B_m\in V$ for each $m<\omega$. Letting $F\colon \omega\times\omega\rightarrow\omega$ be the G\"odel pairing function, define $B^\ast_{F(m,n)}:=B_m\cap\nu_n$, for each $m, n<\omega$. For each $i<\omega$, stipulate $F^{-1}(i)=(m,n)$. It is clear that $f\restriction B^\ast_i=(f\restriction B_m)\restriction \nu_n\in V$, and so $\langle B_i^\ast: i<\omega\rangle\in V[G]$ is a suitable witness for the $\kappa$-trace property.
 \end{proof}
Let us conclude this section with a summarizing table of the relevant definitions and their respective characterizations. 

\begin{table}[htpb]
    \centering
    \renewcommand{\arraystretch}{1.6} 
    
    \begin{tabular}{
        >{\raggedright\arraybackslash}p{2.5cm} 
        >{\raggedright\arraybackslash}p{4.6cm} 
        >{\raggedright\arraybackslash}p{4.6cm}
    }
        \toprule
        \textbf{Property of $\mathbb{P}$} & \textbf{Characterization in $V[G]$} & \textbf{Combinatorial Characterization} \\
        
        $(\omega,\kappa)$-predistributive & 
        $\forall f \in {}^{<\kappa}\text{On}\cap V[G]$, $\exists \mathcal{B} \in V\cap {}^\omega P(\gamma)$ such that $\gamma = \bigsqcup \mathcal{B}$ and $f \upharpoonright \mathcal{B}_n \in V$ for all $n$ & 
        $\forall \vec{\mathcal{D}} \in {}^{<\kappa}\mathcal{F} (\mathbb{Q}) \ \forall p \ \exists q \leq p$ $\exists \vec{E} \in {}^\omega\mathcal{F}_q(\mathbb{Q})$ such that $ \ \forall \alpha \ \exists n \ \vec{E}_n \subseteq \vec{D}_\alpha$ \\
        
        Weakly $(\omega,\kappa)$-predistributive & 
        $\forall f \in {}^{<\kappa}\text{On}\cap V[G]$, $\exists \mathcal{B} \in V[G]\cap {}^\omega P(\gamma)$ such that $\gamma = \bigsqcup \mathcal{B}$ and $f \upharpoonright \mathcal{B}_n \in V$ for all $n$ & 
        $\forall \vec{\mathcal{D}}\in {}^{<\kappa}\mathcal{F}(\mathbb{Q}) \  \exists \vec{A} \in {}^\omega P(\mathbb{Q})$ such that $\vec{A}_n$ is an antichain and $\forall \alpha \ \vec{D}_\alpha \cap \bigcup_n \vec{A}_n$ is predense \\
        
        $\kappa$-trace property & 
        $\forall f \in {}^{\kappa}\text{On}\cap V[G]$, $\exists \mathcal{B} \in V[G]\cap {}^{\omega}P_{\text{bd}}(\kappa)$ such that $\kappa=\bigsqcup\mathcal{B}$ and $ f \upharpoonright \mathcal{B}_n \in V$  for all $n$ & 
        Weakly $(\omega,\kappa^+)$-predistributive and $\one\Vdash_{\mathbb{P}}\cf(\kappa)=\omega$
        \\
        
        \bottomrule
    \end{tabular}
\end{table}

\section{The Tukey type of generic filters and large projection}\label{sec: Tukey type}
In this section we will extend our analysis using the cofinal structure of generic filters. The cofinal structure refers to the study of cofinal subsets of a given partially ordered set\footnote{If $\mathbb{P}$ is a partially ordered set, then a subset $\mathcal{B}\subseteq \mathbb{P}$ is called \textit{cofinal} if for every $p\in\mathbb{P}$ there is $b\in \mathcal{B}$ such that $p\leq_{\mathbb{P}} b$.}. This structure will emphasize the difficulty one has to face when trying to generalize the results of this paper to the class of forcings which are larger than $\lambda$. We will then provide some partial results about this larger class. The main quantity related to the cofinal structure of an ordered set $\mathbb{P}$ is the \textit{cofinality of $\mathbb{P}$} (also known as the \textit{character of $\mathbb{P}$}):
$$\cf(\mathbb{P})=\min\{|\mathcal{B}|\mid \mathcal{B}\subseteq \mathbb{P}\text{ is cofinal in }\mathbb{P}\}.$$ In what follows we view forcing notions as partially ordered sets with their reverse order. Note that cofinal sets coincide with dense sets. We also order filters of the forcing with the order induced from the forcing, which makes them \textit{directed}. Let us immediately start with the main observation of this section:
\begin{theorem}\label{Tukey-type omega Theorem}
    Let $\mathbb{Q}$ be any forcing. Then the following are equivalent:
    \begin{enumerate}
        \item There is a dense $D\subseteq\mathbb{P}_{\mathcal{U}}$ and a forcing projection $\pi:D\to\mathbb{Q}$ which only depends on the stem, i.e, for any $(s,S)\in D$ and any $\U$-tree $T$ with stem $s$, $\pi(s,S)=\pi(s,T)$.
        \item There is a $\mathbb{P}_{\mathcal{U}}$-name $\dot{H}$ such that $\one_\P\Vdash ``\dot{H}\subseteq\mathbb{Q}\text{ is V-generic}, \ \cf(H)=\omega"$,  and for any $q\in\mathbb{Q}$ there is $p\in\mathbb{P}_{\mathcal{U}}$ such that $p\Vdash q\in \dot{H}$.
        \item $\Q$ has a dense subset of size at most $\lambda$ and there is $D\subseteq\mathbb{P}_{\mathcal{U}}$ dense and a projection $\pi:D\to\mathbb{Q}$.
        
    \end{enumerate}
\end{theorem}
\begin{remark} By Theorem \ref{Thm: more general equivalence}, and Fact \ref{fact: proj and complete emb} the above is equivalent to $\mathcal{F}(\mathbb{Q})\leq^{\fin}_{sK}\mathcal{U}$.
\end{remark}
\begin{proof}
    To see $(1)\Rightarrow(3)$, we note that when there is a projection then the image is dense in $\Q$, and since the projection depends only on the stem, the image is of size $\leq\lambda$. 
    To see that $(3)\Rightarrow (2)$ we use Theorem \ref{thm: omega generated from projection} and let $\dot{H}$ a name for $\pi``\dot{G}$. 
    Finally to see $(2)\Rightarrow (1)$, suppose that $\one_{\P_{\mathcal{U}}}\Vdash``\dot{q}_0\geq \dot{q}_1\geq \dot{q}_2...$ generates $\dot{H}$".   
    Now fix an arbitrary $p_*\in \P_{\mathcal{U}}$. 
    For each $n$, find $m_n$ and  $p_n \leq^* p_*$ such that every $m_n$-step extension of $p_n$ decides $(\dot{q}_0,...,\dot{q}_n)$. 
    Let $p^*=(\langle\rangle,A^*)\leq^* p_n$ for all $n$. 
    We will define a projection $\pi:\mathbb{P}_{\mathcal{U}}/p^*\to \mathbb{Q}$ which depends only on the stem.
    We will start by defining $\pi \restriction \bigcup_{n<\omega}[A^*]^{m_n}$, then extend to $\pi \restriction [A^*]^{<\omega}$ by setting $\pi(\vec{x})=\pi(\vec{x}\restriction m_n)$ for the largest\footnote{We may assume that $\dot{H}$ is forced to be not in the ground model. This ensures the existence of such a largest $n$. Also, by thinning out if necessary, we can assume that $m_{n+1}>m_{n}$.} $n$ such that $m_n \leq |\vec{x}|$.
    $\pi(\vec{x})$ for $\vec{x}\in [A^*]^{m_n}$ to be $q$ such that $p^{*\smallfrown}\vec{x}\Vdash \dot{q}_n=q$. 
    If $\vec{x}\sqsubseteq \vec{y}$, since the values of $\dot{q}_n$'s are forced to be decreasing, $\pi(\vec{x})\geq \pi(\vec{y})$. Let $D\subseteq\mathbb{Q}$ be dense open, and let $(\vec{x},B)\leq (\langle\rangle,A^*)$. Then there is $(\vec{y},C)\leq (\vec{x},B)$ and some $N$ such that $(\vec{y},C)\Vdash \forall n\geq N, \dot{q}_n\in \check{D}$. Then pick any $\vec{y}\sqsubseteq\vec{z}$ such that $|\vec{z}|\geq m_N$ and $\vec{z}\setminus \vec{y}\subseteq C$. Then $\pi(\vec{z})\in D$. Indeed, $\pi(\vec{z})$ is forced by $(\vec{z},C\setminus\max(\vec{z})+1)$ to be $\dot{q}_m$ for some $m\geq N$ but also the same condition forces that $\dot{q}_m\in D$ so $\pi(\vec{z})\in D$. Hence $\pi$ is a weak projection. Since we started with an arbitrary $p_*$, we can now define a weak projection $\pi$ from a dense subset $D\subseteq \mathbb{P}_{\mathcal{U}}$ to $\mathbb{Q}$ (by defining the projection along an antichain). To see that this is a projection, it suffices to show that $\rng(\pi)$ is dense. Indeed, let $q\in\mathbb{Q}$, then there is $p$ such that $p\Vdash q\in\dot{H}$. Then we can find $p'\leq p$ and $n$ such that $p'\Vdash \dot{q}_n\leq q'$ (indeed, the $\dot{q}_n$'s are forced to generate $\dot{H}$). We can further extend $p'$ to $p''$ and find $q^*$ such that $p''\Vdash q^*=\dot{q}_n$ and therefore $q^*\leq q$.
    Find $d\in D$ such that $d\leq p''$ and make sure that the stem of $d$ has size at least $m_n$, Then $\pi(d)\leq q$.
\end{proof}

\begin{cor}
    Let $\mathbb{Q}$ be a $(\omega,\kappa)$-predistributive $\lambda$-pregenerated forcing. Then there is  $D\subseteq \mathbb{P}_{\mathcal{U}}$ dense, and there is a projection $\pi:D\to\mathbb{Q}$ which only depends on the stem.
\end{cor}

\begin{example}
    For each pair of the following forcing notions there is a projection which only depends on the stem. 
    \begin{enumerate}
        \item $\mathbb{P}_{\mathcal{U}}$ and $\mathbb{P}(\mathcal{W})$, where $\mathcal{W}$ is a measure on $\kappa$ and $\mathbb{P}(\mathcal{W})$ denotes vanilla Prikry forcing.
        \item $\mathbb{P}_{\mathcal{U}}$ and $\mathbb{P}_{\mathcal{U_\alpha}}$, where $\kappa\leq\alpha<\lambda$ and $\mathcal{U}_\alpha$ is the projection of $\mathcal{U}$ to $\mathcal{P}_\kappa(\alpha)$. Note that the usual weak projection $(\langle s_0,\dots, s_n\rangle, A)\mapsto(\langle s_0\cap\alpha,\dots, s_n\cap\alpha\rangle, \{X\cap\alpha: X\in A\})$ depends on the measures.
        \item A cone of AIM forcing $\mathbb{Q}/q$ and $\mathbb{Q}_{\vec{\beta}}$, where $\Vec{\beta}=\langle\beta_i:\ell(q)\leq i<\omega\rangle$ is a $\leq$-increasing sequence with $\beta_i\in a_i^q$ for all $i\geq \ell(q)$. Note that the usual weak projection $\pi:\mathbb{Q}/q\to\mathbb{Q}_{\vec{\beta}}$, given by $\pi(p)_i=f_i^p(\beta_i)$ for $\ell(q)\leq i<\ell(p)$, and $\pi(p)_i=\{x\cap\beta_i: x\in A_i^p\}$ for $i\geq \ell(p)$, depends on the measures (see \cite{cummings2018ordinal}).
        \item Merimovich forcing $\mathbb{P}$ and $\mathbb{P}_d$, with $d \in [\lambda \setminus \kappa]^{<\lambda}$ and $\kappa = \min(d)$ (see \cite[$\S4$]{PovedaThei2024baire}).
    \end{enumerate}
    \end{example}

The above theorem illustrates how the ``shape" of a generic for a forcing $\mathbb{Q}$ of cardinality $\lambda$ which is a projection of $\mathbb{P}_{\mathcal{U}}$ is as simple as it can be: it is generated by an $\omega$-sequence of conditions.
This is no longer true when we consider forcings of larger cardinality; in which case, more complicated shapes of generating sets appear. 
To formally express what a ``shape'' is, it is natural to use the language of the Tukey order:
\begin{definition}[Tukey Order]\label{Def: Tukey}
    Let $(P,<_P)$ and $(Q,<_Q)$ be two directed partial orders.
    We say that $P \leq_T Q$ iff there is a map $f : Q \to P$ such that whenever $A \subseteq Q$ is cofinal in $<_Q$ then $f``A$ is cofinal in $<_P$.
    We say $P \equiv_T Q$ iff $P \leq_T Q \leq_T P$.
\end{definition}
Theorem \ref{Tukey-type omega Theorem} says that if $\Q$ is a projection of $\mathbb{P}_{\mathcal{U}}$ of cardinality $\leq\lambda$, then a generic filter $H$ for $\Q$ which has Tukey-type $\omega$ in $V[G]$.
A more general (and trivial) observation is the following:
\begin{fact}
    Let $\mathbb{P},\mathbb{Q}$ be forcing notions and suppose $\pi:\mathbb{P}\to\mathbb{Q}$ is a weak projection.
    Then for any generic $G\subseteq \mathbb{P}$, $G\geq_T \pi``G$.
\end{fact}
This observation provides a useful tool to rule out the existence of weak projections. 
Note that the Tukey-type of $G$ and $\pi``G$ is compared in $V[G]$. 
This invites the study of the cofinal structure of a generic filter, which we will see is very relevant to the structure of intermediate models of Prikry-type forcings. 
Let us start with some simple examples which show that Tukey-comparability of the generic filter is much coarser than the existence of a weak projection:
\begin{example}
    Let $G$ be $V$-generic for $\Add(\kappa,1)$ or $\col(\kappa,\lambda)$, then $G\equiv_T \kappa$: 
    Indeed, let $g=\bigcup G$.
    Then $\langle g\restriction \alpha: \alpha<\kappa\rangle$ is a linearly ordered set which generates $G$, which suffices to conclude that $G\equiv_T \kappa$. 
\end{example}
More generally, if $G$ is $V$-generic for a $\kappa$-distributive forcing and in some model $V[G]\subseteq M$, $|G|=\kappa$, then $M\models G\equiv_T \kappa$. Moreover, if $M\models \cf(\kappa)=\lambda$, then $G\equiv_T\lambda$.
 The next theorem establishes the equivalence of $\kappa$-directedness of the generic filter and $\kappa$-distributivity of the forcing. In particular, any forcing with generics being of Tukey type $\kappa$ must be $\kappa$-distributive.
\begin{theorem}\label{theorem: equivalent directed}
    Let $\mathbb{P}$ be a forcing notion and $\kappa$ be a regular cardinal. The following are equivalent:
    \begin{enumerate}
        \item $\Vdash_{\mathbb{P}}``\dot{G}$ is $\kappa$-directed".
        \item $\mathbb{P}$ is ${<}\kappa$-distributive.
    \end{enumerate}
\end{theorem}
\begin{proof}
    For $(1)\Rightarrow(2)$, let $f:\gamma\to V\in V[G]$ be any function with $\gamma<\kappa$. 
    Fix a name $\dot{f}$ for $f$.
    For each $\beta<\gamma$, there is $p_\beta\in G$ such that $p_\beta\Vdash \dot{f}(\beta)=f(\beta)$. 
    Since $G$ is $\kappa$-directed, there is $p\in G$ such that $p \leq p_\beta$ for all $\beta<\gamma$. 
    So $p$ decides every value of $f$ and therefore $f\in V$. 

    For $(2)\Rightarrow (1)$, suppose $\mathbb{P}$ is ${<}\kappa$-distributive. 
    To see that $G$ is $\kappa$-directed, let $\{p_\alpha: \alpha<\gamma<\kappa\}\subseteq G$. 
    By distributivity, $\{p_\alpha: \alpha<\gamma\}\in V$.
    Hence for each $\alpha < \gamma$ we can define the dense open set $D_\alpha = \{ q \in \mathbb{P} :  p_\alpha \leq q \text{ or } q\perp p_\alpha\}$.
    Again by distributivity, $\bigcap_{\alpha<\gamma}D_\alpha$ is dense open so there is $p\in G\cap \bigcap_{\alpha<\gamma}D_\alpha$. 
    Clearly $p$ is compatible with each $p_\alpha$, so $p \leq p_\alpha$ for all $\alpha < \gamma$.
\end{proof}
We say that a directed set is \textit{$(\omega,\kappa)$-directed closed} if for any $\gamma<\kappa$ and any sequence $\langle d_i: i<\gamma\rangle\subseteq D$ there is another sequence $\langle e_n: n<\omega\rangle$ such that for every $i<\gamma$ there is $n<\omega$ such that $d_i\leq e_n$. 
\begin{theorem}
    Let $\mathbb{P}$ be a forcing notion and $\kappa$ be a regular cardinal. The following are equivalent:
    \begin{enumerate}
        \item $\Vdash_{\mathbb{P}}``\dot{G}$ is $(\omega,\kappa)$-directed closed".
        \item $\mathbb{P}$ is weakly $(\omega,\kappa)$-predistributive.
    \end{enumerate}
\end{theorem}
\begin{proof}
    To show, $(1)\Rightarrow(2)$ let $\gamma < \kappa$ and $f:\gamma\to V$ be any function in $V[G]$.
    Fix a name $\dot{f}$ for $f$ and a $p$ such that $p \forces \dot{f} : \gamma \to V$. 
    For each $\beta<\gamma$, there is $p_\beta\in G/p$ such that $p_\beta\Vdash \dot{f}(\beta)=f(\beta)$. 
    Since $G$ is $(\omega,\kappa)$-directed closed, there is a sequence $\langle q_n: n<\omega\rangle\subseteq G$ such that for each $\beta$ there is $n$ such that $q_n\leq p_\beta$. 
    For each $n < \omega$ let $B_n=\{\beta: q_n ||\dot{f}(\beta)\}$. 
    Then $B_n\in V$ and $\langle B_n: n<\omega\rangle\in V[G]$ is a partition of $\gamma$ and $f\restriction B_n\in V$ for every $n \in \omega$. 
    Hence $\mathbb{P}$ is weakly $(\omega,\kappa)$-predistributive.
    
    For $(2)\Rightarrow (1)$, suppose $\mathbb{P}$ is weakly $(\omega,\kappa)$-predistributive. To see that $G$ is $(\omega,\kappa)$-directed let $\gamma<\kappa$ and $\langle p_\alpha: \alpha<\gamma\rangle\subseteq G$. 
    By weak $(\omega,\kappa)$-predistributivity, there is a partition $\langle B_n: n<\omega\rangle$ of $\gamma$ so that for each $n$, $\langle p_\alpha: \alpha\in B_n\rangle\in V$. this means that there is a condition $q_n\in G$ that forces $\langle p_\alpha: \alpha\in B_n\rangle\subseteq G$. 
    By separativity, this means that $q_n \leq p_\alpha$ for every $\alpha\in B_n$. 
    Hence the sequence $\langle q_n: n<\omega\rangle\subseteq G$ is as wanted. 
\end{proof}

To characterize the forcings whose generic is exactly of Tukey-type $\kappa$, for a regular cardinal $\kappa$, we observe that $D\equiv \kappa$ if and only if $D$ is $\kappa$-directed and generated by $\kappa$-many sets. 
For a ordered set $\mathbb{P}$, define the cofinality of $\mathbb{P}$, denoted by $\cf(\mathbb{P})$, to be the minimal cardinality of a generating set for $\mathbb{P}$. Since we already characterized posets whose generic is $\kappa$-directed (i.e. $\kappa$-distributive), it remains to characterize when is it forced that $\cf(\dot{G})\leq\kappa$. This is related to the notion of $\kappa$-pregenerated forcings in Definition \ref{def: pregenerated}:
\begin{prop}
    If $\mathbb{\Q}$ is  $\lambda$-pregenerated and $\dot{G}$ names a $\mathbb{Q}$-generic filter then $\Vdash_{\mathbb{Q}}\cf(\dot{G})\leq\lambda$.
\end{prop}
\begin{proof}
Suppose $\mathcal{F}(\Q)$ is generated by $\{ B_\alpha : \alpha < \lambda \}$.
By shrinking the $B_\alpha$'s if necessary, we may assume that they are all open dense.
Suppose $G \subseteq \mathbb{\Q}$ is generic.
For each $\alpha$ choose some $q_\alpha \in G \cap B_\alpha$.
Let $X_0 = \{q_\alpha : \alpha < \lambda \}$.
We extend $X_0$ to a directed set of size $\lambda$ inductively.
Define a function $F_1 : [X_0]^{<\omega} \to G$ where $F_1(s)$ is a lower bound of $s$ and $F_1(s) \in X_0$ if possible.
Set $X_1 = F_1 `` X_0$.
Define $X_{n+1}$ and $F_{n+1} : [X_n]^{<\omega} \to G$ similarly so that $F_{n+1} \upharpoonright [X_{n-1}]^{<\omega} = F_n$.
Let $X^* = \bigcup_{n \in \omega} X_n$.
Then $X^*$ is directed and meets every dense subset of $\Q$ so the upwards closure of $X^*$ is a generic filter.
This upwards closure is contained in $G$, so in fact $X^*$ generates $G$, and $|X^*| = \lambda \cdot \omega  = \lambda$.
So $G$ is generated by a set of size $\lambda$.
\end{proof}

\begin{remark}
    The above implication is not reversible:
Consider the Cohen forcing $\Q = \add(\omega,1)$ and let $G$ be generic, and $f = \bigcup G$.
$G$ is generated by a countable set: $\{ f \upharpoonright n: n \in \omega \}$.
However, no non-trivial forcing can be $\omega$-pregenerated, since this would allow us to construct a generic in the ground model. 
\end{remark}
\begin{question}
    \ {} \begin{enumerate}
        \item Is there a characterization of $\lambda$-pregenerated forcings in terms of the Tukey-type of the generic filter? 
        \item What is the combinatorial characterization of posets $\mathbb{Q}$ such that $\Q\Vdash \cf(\dot{G})\leq\check{\kappa}$?
    \end{enumerate}
\end{question}
To see examples where we have more complicated Tukey-types, we need the following characterization:
\begin{lemma}[Tukey \cite{Tukey}]\label{lemma: equivalent condition}
Let $\lambda$ and $\kappa$ be cardinals such that $\cf([\lambda]^{<\kappa},\subseteq)=\lambda$. Then the following are equivalent for every directed set $\mathbb{P}$:
\begin{enumerate}
    \item $[\lambda]^{<\kappa}\leq_T \mathbb{P}$.
    \item For every $\kappa$-directed order $\mathbb{D}$ of size $\leq\lambda$, $\mathbb{D}\leq_T\mathbb{P}$.
    \item There is $\langle p_\alpha\mid \alpha<\lambda\rangle\subseteq \mathbb{P}$ such that for every $I\in[\lambda]^\kappa$, $\{p_\alpha\mid \alpha\in I\}$ is unbounded in $\mathbb{P}$.
\end{enumerate}
\end{lemma}
\begin{example}
    If $G$ is generic for either $\Add(\kappa,\lambda)$ or $\col(\kappa,<\lambda)$, then in the extension $G\equiv_T [\lambda]^{<\kappa}$. Note that $|G|=\lambda$ in both cases, and is $\kappa$-directed (by Theorem \ref{theorem: equivalent directed}). Moreover, any $\kappa$-many elements of $G$ are unbounded (even in the forcing) since conditions are of size ${<}\kappa$. By Lemma \ref{lemma: equivalent condition}, $G\equiv([\lambda]^{<\kappa},\subseteq)$.  
\end{example}
Next, let us analyze the Tukey-type of Prikry forcing generics, or on $\omega$, those of Mathias and Laver forcings.
The definitions of Mathias and Laver forcing can be found in \cite{Jech2003}.
\begin{prop}
    Let $F$ be a filter over $\kappa$ for $\kappa\geq\omega$.
    Let $G \subseteq \Q$ be generic where $\Q$ be either:
    \begin{enumerate}
        \item Prikry forcing with $F$.
        \item Laver forcing relative to a filter (i.e. splittings are $F$-measure one) (or Tree Prikry forcing relative to $F$).
        \item Mathias forcing relative to $F$
    \end{enumerate}
   Then $G\equiv_T F$.
\end{prop}
\begin{proof}
All these forcings have the following in common:
\begin{enumerate}
    \item If $G$ is $V$-generic for $\Q$, then $G$ induces a sequence $\vec{c}_G$ (of some length $\omega\leq \ell\leq\kappa)$ such that $G=\{(\vec{c}_G\restriction \gamma,T)\in \mathbb{Q}: \gamma<\ell, \ \vec{c}_G\restriction [\gamma,\ell)\subseteq T\}$
    \item $A\in F$ if and only if there is $\gamma<\ell$ such that $\vec{c}\restriction [\gamma,\ell)\subseteq A$.
\end{enumerate}

    For the Mathias forcing relative to a filter $F$, or the normal Prikry forcing, recall that if $\langle c_n: n<\omega\rangle$ is generic then $A\in F$ if and only if there is $N<\omega$ such that $\forall n\geq N$, $c_n\in A$. Define in $V[G]$, $f:F\to G$ by $f(A)=(c_0,...,c_{N_A},A)$ where $N_A$ is the list such that for all $n\geq N_A$, $c_n\in A$. Then $f$ is well defined, monotone, onto and also unbounded, since if $\mathcal{A}\subseteq F$ is unbounded, then $\bigcap\mathcal{A}\notin F$ and $\{f(A): A\in\mathcal{A}\}$ is unbounded (even in $\mathbb{P}$) since any stronger condition will give a bound for $\bigcap\mathcal{A}$.
\end{proof}
 We get the following corollary (below $Pr(U)$ denotes the classical Prikry forcing with $U$):
\begin{cor}
    If $Pr(U)$ weakly projects to $Pr(W)$, then in $V^{Pr(U)}$, $U\geq_T W$.
\end{cor}
\begin{cor}\label{cor: shape}
    If $\mathbb{P}_{\mathcal{U}}$ weakly projects to $\mathbb{Q}$, then $\Vdash_{\mathbb{P}_{\mathcal{U}}} \mathcal{U}\geq_T \dot{G}_{\mathbb{Q}}$
\end{cor}
But what is the Tukey-type of $\mathcal{U}$ in $V[G]$? In \cite{TomBen} it was shown that every fine $\kappa$-complete ultrafilter over $P_\kappa(\lambda)$, where $\lambda$ is regular, has maximal Tukey complexity, i.e. $\mathcal{U}\geq_T [2^\lambda]^{<\kappa}$.
\begin{prop}
    Suppose that  $G$ is $V$-generic for $\mathbb{P}_{\mathcal{U}}$. 
    The following holds of $\mathcal{U}$ in $V[G]$:
\begin{enumerate}
    \item $\mathcal{U}=\bigcup_{n<\omega}\mathcal{U}_n$ such that each $\mathcal{U}_n$ is $V$-complete in the sense that for any $\mathcal{A}\in V\cap P(\mathcal{U})$, $\bigcap\mathcal{A}\in\mathcal{U}$.
    \item Let $\lambda$ be regular in $V[G]$. Then $V[G]\models \mathcal{U}\geq [\lambda]^{<\kappa}$ if and only if $V\models \mathcal{U}\geq [\lambda]^{<\kappa}$.
    \item If $\mathcal{U}$ is normal, then in $V[G]$, $\mathcal{U}$ is $\nu$-indecomposable for every regular uncountable $\nu<\kappa$.
\end{enumerate}
\end{prop}
\begin{proof}
    For $(1)$ simply take $\mathcal{U}_n=\{A\in \mathcal{U}: \forall m\geq n, (\vec{c}_G)_m\in A\}$ where $\vec{c}$ is the Prikry sequence.
    
    For $(2)$, if $\mathcal{U}\geq [\lambda]^{<\kappa}$ in $V$, then there is a sequence witnessing this, say $\langle A_\alpha: \alpha<\lambda\rangle\in V\cap [\mathcal{U}]^\lambda$. 
    We claim that the sequence remains a witness in $V[G]$.
    Indeed, if $I\in [\lambda]^\kappa\cap V[G]$, and $\bigcap_{i\in I}A_i\in\mathcal{U}$, let $A\in \mathcal{U}$ be such that $A\subseteq \bigcap_{i\in I}A_i$.
    Then $A\in V$ and we may define $I'=\{i<\lambda: A\subseteq A_i\}\in V$.  
    Since $I\subseteq I'$, then $I'\in [\lambda]^\kappa\cap V$, contradicting the fact that $\langle A_\alpha:\alpha<\lambda\rangle $ is a witness. 
    In the other direction, if $\langle A_\alpha: \alpha<\lambda\rangle\in V[G]\cap [\mathcal{U}]^\lambda$ is a witness, then for every $\alpha$ we find $(t_\alpha,B_\alpha)\Vdash \dot{A}_\alpha=A'_\alpha$ (below a condition which forces everything). 
    By the regularity of $\lambda$, there is $t^*$ and $J\in [\lambda]^\lambda$ such that for every $\alpha\in J$, $t_\alpha=t^*$. Let $A^*_\alpha=B_\alpha\cap A'_\alpha$. 
    We claim that $\langle A^*_\alpha: \alpha\in J\rangle$ is a witness in $V$: 
    Otherwise, there is $I\in [J]^\kappa$ such that $A^*:=\bigcap_{\alpha\in I}A^*_\alpha\in U$. 
    But then $\langle t^*,A^*\rangle$ forces for each $\alpha\in I$ that $\dot{A}_\alpha=A'_\alpha$ and therefore forces that $A^*\subseteq \bigcap_{\alpha\in I}\dot{A}_\alpha$. 
    But then $\langle t^*,A^*\rangle$ forces contradictory information. 
    
    Finally, for $(3)$, see for example \cite[Ex. 18.6]{kanamori1994}.
\end{proof}
From $(2)$ we conclude that Prikry forcing with
a $p$-point ultrafilter (or any other non-Tukey-top ultrafilter) cannot add a Prikry sequence for a Tukey-top ultrafilter.
\begin{cor}
    If $Pr(U)$ weakly projects to $Add(\kappa,\lambda)$, then $U\geq_T [\lambda]^{<\kappa}$.
\end{cor}
Let us finish this paper by several observation regarding the forcing projection of $\P_{\mathcal{U}}$ if larger cardinality. 
First, we have the following:
\begin{prop}
    Suppose that $\kappa$ is supercompact, $\Q$ is a $\lambda^+$-cc forcing, $(\omega,\kappa)$-predistributive  and not $\kappa$-distributive. Then $\Q$ adds an $\omega$-sequence unbounded in some $V$-cardinal in the interval $[\kappa,\lambda]$. In particular, the same holds for non-distributive projections of $\P_{\mathcal{U}}$. 
\end{prop}
\begin{proof}
  By Corollary \ref{distributive projections}, $\Q$ must add an  $\omega$-sequence  $\langle \alpha_n\mid n<\omega\rangle$. By the chain condition, we can cover $\{\alpha_n\mid n<\omega\}$ by some set $A\in V$, $|A|\leq \lambda$. By properly enumerating, we see that an $\omega$-sequence was added below $\lambda$. Let $\lambda'\leq\lambda$ be minimal such that some $\omega$-sequence was added below $\lambda'$, then by minimality, $\lambda'$ must be a $V$-cardinal, and the new $\omega$-sequence must be unbounded in $\lambda'$.  Since $\kappa$ is supercompact, $\Q$ must also be $(\kappa,2)$-distributive, and therefore $\lambda'\geq\kappa$.
\end{proof}
\begin{prop}
    $\add(\kappa,\kappa^{++})$ is not a weak projection of a $P_\kappa(\kappa^+)$-strongly compact Prikry forcing.
\end{prop}
\begin{proof}
    Otherwise, let $A\subseteq \kappa^{++}$ be $\add(\kappa,\kappa^{++})$-generic (i.e. $<\kappa$-approximations are in $V$ and if $I\subseteq\kappa^{++}$ is of size $\geq\kappa$, then $A\restriction I\notin V$). In $V$, for each $\alpha<\kappa^{++}$, let $p_\alpha=(t_\alpha,B_\alpha)$ decide $\dot{A}(\alpha)$. Then there is $I\in [\kappa^{++}]^{\kappa^{++}}$ such that $t_\alpha=t^*$ for every $\alpha\in I$. Let $B=\{\alpha: \exists C, \ (t^*,C)|| \dot{A}(\alpha)\}$. Note that the decision $\dot{A}(\alpha)$ does not depend on $C$, hence $B$ is a set in $V$ of size $\kappa^{++}$ which is forced to agree with $\dot{A}$. Contradiction.
\end{proof}
The above proof shows something more general:
If a forcing of size $\kappa^{++}$ is absorbed, then the generic is just a (characteristic) function on $\kappa^{++}$, $f\in V[G]$. 
Then there is a partition $\langle B_n: n<\omega\rangle$ of $\kappa^{++}$ in $V[G]$ such that $f\restriction B_n\in V$ for every $n$. 
Note that this is difference from weak $(\omega,\kappa)$-predistributivity since the proof only works for binary functions (i.e. sets). 
\begin{cor}\label{larger projections}
    Suppose that $\mathcal{U}$ is a fine $\kappa$-complete ultrafilter over $P_\kappa(\kappa^+)$. Then $\mathbb{P}_{\mathcal{U}}$ cannot weakly project to $\add(\kappa^+,\kappa^{++})$ nor to $\add(\kappa^{++},1)$. Moreover, $\mathbb{P}_{\mathcal{U}}$ cannot add fresh subsets to $\kappa^{++}$.
\end{cor}
\section{Questions and concluding discussion}\label{section questions}
We summarize our contribution to the understanding of the intermediate models of the strongly compact Prikry forcing and the remaining tasks, together with other related research avenues.

Most of our results classify the subforcings of size $\lambda$ of the strongly compact Prikry forcing with an ultrafilter on $\pkl$. We have seen that:
\begin{enumerate}
    \item [(i)] Given that $\kappa$ is supercompact, the class of forcings which are projections of a strongly compact Prikry forcing coincide with the class of $(\omega,\kappa)$-predistributive forcings. 
    \item [(ii)] Inside the class of forcings of size $\leq\lambda$, the class of forcings which can be projections of a $\lambda$-strongly compact Prikry forcing on $\pkl$ is exactly the class of $(\omega,\kappa)$-predistributive forcings.
    \item [(iii)] Inside the class of forcings of size $\leq\lambda$, the class of projections of every strongly compact Prikry forcings on $\pkl$ include all the $(\omega,\kappa)$-predistributive $\lambda$-pregenerated forcings.
    \item [(iv)] A projection of the $\lambda$-strongly compact Prikry forcing must:
    \begin{enumerate}
        \item have a dense subset of size $2^\lambda$.
        \item be $(\kappa,\lambda)$-centered.
        \item be $(\omega,\kappa)$-predistributive.
    \end{enumerate}
\end{enumerate}
This suggests the following questions:
\begin{question}
    If $\mathcal{U}$ is a fine, $\kappa$-complete ultrafilter over $\pkl$, which forcings of size $> \lambda$ are projections of $\P_\mathcal{U}$?
\end{question}
We have ruled out some classical ones in \ref{larger projections}, and item (iv) provides some restriction. A related question is the following (formulated for $\lambda=\kappa^+$ for concreteness, but of interest for any other $\lambda>\kappa$):
\begin{question}
    Suppose that $\mathbb{Q}$ is a projection of the $\kappa^+$-strongly compact Prikry forcing, which does not have a dense subset of size $\kappa^+$. Must $\mathbb{Q}$ singularize $\kappa^+$?
\end{question}
Regarding uncountable cofinality, we should be interested in the Magidor-Radin supercompact forcing from \cite{SupercompatRadinExtender}:
\begin{question}
    Which forcings are consistently projections of the Magidor-Radin supercompact forcing?
\end{question}
Some initial results can be found in \cite{TomAle}. 

We believe that central problems arising from our work regard the supercompact Prikry forcing. We proved limited results regarding $(\omega,\kappa)$-predistributive $\lambda$-pregenerated forcings in Corollary \ref{cor for supercompact}.
\begin{question}
    What forcings are consistently projections of the $\lambda$-supercompact Prikry forcing? Which are projections of every $\lambda$-strongly compact Prikry forcing? 
\end{question}
Some rigidity is expected in analogy to Gitik-Kanovei-Koepke's result \cite{PrikryCaseGitikKanKoe} for the normal Prikry forcing (on $\kappa$).
The following should be the first test questions on the supercompact case, since we have positive answers in the strongly compact case. More specifically we ask:
\begin{question}
    What forcings of size $\leq\lambda$ are projections of the $\lambda$-supercompact Prikry forcings?\end{question}
    More concretely: \begin{question}Can $\add(\kappa^+,1)$ consistently be a projection of $\P_\U$ for $\U$ a normal measure over $P_\kappa(\kappa^+)$? How about $\Col(\kappa,\kappa^+)$?
\end{question}
Our results show that these forcings are always projections of $\P_\U$ for $\U$ a $\kappa^+$-strongly compact measure.

In section $5$ we analyze the existence of projections via the cofinality and other order theoretic properties of the generic filter.
We believe this analysis can be extended much further, providing avenues for future research.

\emph{\small Acknowledgments.
The first author's work is supported by the National Science Foundation under Grant
No. DMS-2346680. The first author would like to thank Moti Gitik, Gabriel Goldberg, Yair Hayut and Andreas Lietz for interesting discussions on topics related to the work in this paper. The second author's work is funded by the National Science Center, Poland under the Weave-Unisono Call, registration number UMO-2023/05/Y/ST1/00194. This work was partially supported by the Simons Foundation grant (award no. SFI-MPS-T-Institutes-00010825) and from State Treasury funds as part of a task commissioned by the Minister of Science and Higher Education under the project “Organization of the Simons Semesters at the Banach Center - New Energies in 2026-2028” (agreement no. MNiSW/2025/DAP/491).}
\begin{center}   
        \includegraphics[width=0.8\linewidth]{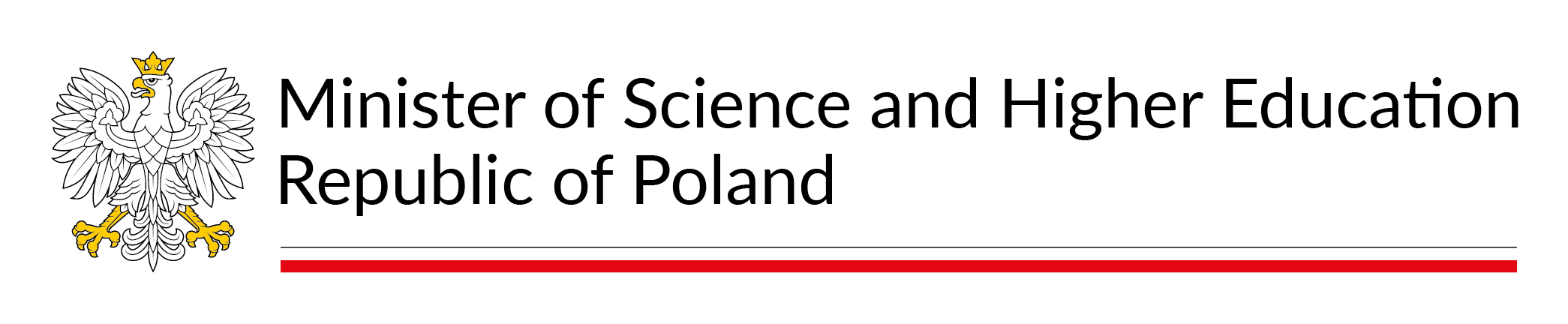}
    \end{center}

    \begin{center}   
        \includegraphics[width=0.8\linewidth]{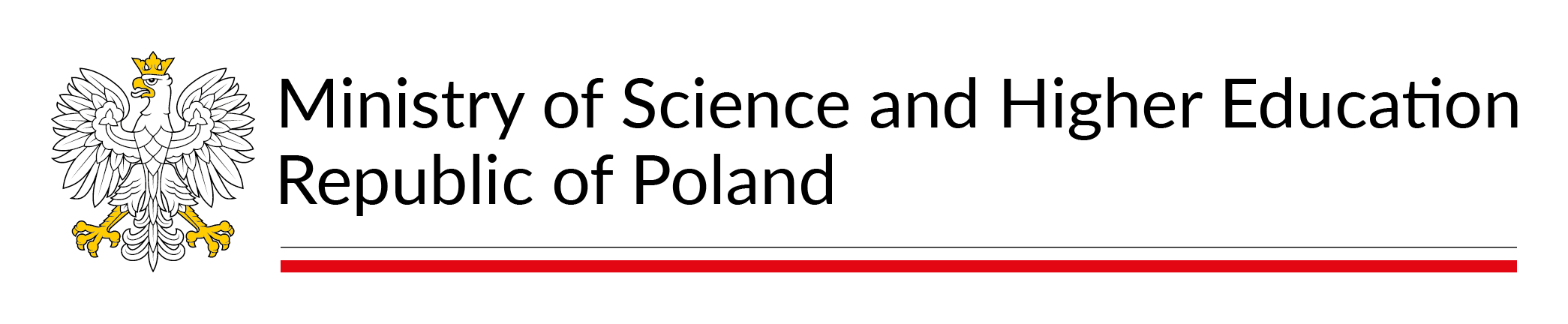}
    \end{center}
    \bibliographystyle{alpha}
\bibliography{biblio}
\end{document}